\documentclass[twoside,a4paper,12pt,reqno]{amsart}
\usepackage{amsfonts, amsbsy, amsmath, amssymb, latexsym}
\usepackage{mathrsfs,array}
\usepackage[top=1 in,right=1 in,bottom=1 in,left=1 in]{geometry}
\usepackage[utf8]{inputenc}
\usepackage{tikz}

\usepackage{mathpartir}
\usepackage{etoolbox}
\usepackage{faktor}

\usepackage{graphicx}
\usepackage{mathtools}

\usepackage{floatflt}
\usepackage{multicol}
\usepackage{lipsum}

\usepackage{diagbox}

\usepackage{wrapfig}
 \usepackage{relsize}
 
\usepackage{booktabs}
\usepackage{pgfplots}
\usepackage{pgfplotstable}
\usepackage{enumerate}

\usepackage{blindtext}
\usepackage{amssymb}
\usepackage{amsthm}
\usepackage{amsmath}

\usepackage{siunitx}

\newcommand{\N}{\mathbb{N}}

\newcommand{\G}{\Gamma}

\newcommand{\ra}{\rangle}
\newcommand{\la}{\langle}
\newcommand{\ma}{\mathcal{O}_}
\newcommand{\te}{\text}
\usetikzlibrary{lindenmayersystems,arrows.meta}

\usepackage{array,tabularx}

\newtheorem{thm}{Theorem}
\newtheorem*{conj*}{Conjecture}

\newtheorem{theorem}{Theorem}[section]
\newtheorem{proposition}[theorem]{Proposition}
\newtheorem{lemma}[theorem]{Lemma}
\newtheorem{corollary}[theorem]{Corollary}

\newtheorem{dfn}[theorem]{Definition}

\theoremstyle{definition}

\newtheorem{construction}[theorem]{Construction}
\newtheorem{example}[theorem]{Example}
\usepackage{lipsum}

\newcommand{\Aut}{\mathrm{Aut}}
\newcommand{\Dic}{\mathrm{Dic}}
\newcommand{\Dih}{\mathrm{Dih}}

\usepackage{cite}
\usepackage{blindtext}
\begin{document}
\title{Smallest graphs with given automorphism group
}
\author{DANAI DELIGEORGAKI}
\address{
Department of Mathematics, KTH, se-100 44 Stockholm, Sweden}
\email{danaide@kth.se}

\vspace{-2 mm}

\begin{abstract}\vspace{-1 mm}
For a finite group $G$, denote by $\alpha(G)$
the minimum number of vertices of any graph 
$\G$ having $\Aut(\G)\cong G$.
In this paper, we prove that $\alpha(G)\leq |G|$, with specified exceptions. The exceptions include four infinite families of groups, and 17 other small groups. Additionally, we compute $\alpha(G)$ for the groups $G$ such that $\alpha(G)> |G|$ where
the value $\alpha(G)$ was previously unknown.\vspace{-3 mm}
\end{abstract}

\maketitle
\vspace{-4 mm}
\section{Introduction}
\vspace{-1 mm}

In \cite{fru} it is shown that every finite group can be realised, up to isomorphism, as the automorphism group of a finite graph;
in fact, for every finite group $G$ there exist infinitely many finite graphs having automorphism group isomorphic to $G$.
Given a finite group $G$, define $\alpha(G)$ to be the smallest number of vertices of any graph $\G$ having $\Aut(\G)\cong G$. The problem of finding $\alpha(G)$ has been considered by many authors. The value of $\alpha(G)$ has been determined in \cite{arli} for abelian groups $G$, in \cite{gl1,gl,hag,djmc} for dihedral groups $G$, in \cite{last} for quasi-dihedral groups and quasi-abelian groups $G$
and in \cite{grav} for generalised quaternion groups  $G$. The 
question has also been investigated for several families of finite simple groups in \cite{lie}.
A recent survey on this problem can be found in \cite{jess}.
In \cite{Babai}, Babai showed that $\alpha(G)\leq 2|G|$, for every finite group $G$ that is not cyclic of order 3, 4 or 5. 
In this paper, we improve Babai's bound to $|G|$, with specified exceptions (including four infinite families of groups).

In Table \ref{table:kd}, we let
$\Dic_m=\la a,b\mid a^{2m}=1, \;b^2=a^m,\;bab^{-1}=a^{-1} \ra$ for $m=3,5,6$, which is a group of order $4m$, 
and
${G}_{16}=\la a,b \mid a^4=b^4=1,\;bab^{-1}=a^{-1} \ra$,
${G}'_{16}=\la a,b\mid a^{8}=b^{2}=1,\;bab^{-1}=a^{5}\ra$, which are groups of order 16.
\begin{thm}{\label{1}}
Let $G$ be a finite group of order $n$. Then one of the following is true:
\begin{enumerate}[$i)$]
    \item$\alpha(G)\leq n,$
    \item$G$ is cyclic of order $p^k$ or $2p$, where $p$ is prime and $k$ is positive integer $(n\neq 2)$, 
     \item$G$ is $Q_{2^r}$ or $Q_{2^{r}}\times C_2$, where $Q_{2^r}$ is the generalised quaternion group of order $2^r,r\geq 3$,
    \item$G$ is one of the 17 exceptional groups of order at most 25 shown in Table \ref{table:kd}.
\end{enumerate}
If $ii),iii)$ or $iv)$ holds, then $\alpha(G)>n$; 
indeed, if $ii)$ holds,
$\alpha(G)$ is as described in Propositions \ref{27} and
\ref{27b};
if $iii)$ holds,
$\alpha(G)=2n$ or $\alpha(G)=n+2$ for $G=Q_{2^r}$ or $G=Q_{2^{r}}\times C_2$, respectively; if $iv)$ holds,
$\alpha(G)$ is as shown in Table \ref{table:kd}.
\end{thm}
As a consequence, we deduce when equality holds in Babai's bound.
\begin{corollary}Let $G$ be a finite group of order $n$. Then 
$\alpha(G)=2n$ if and only if 
\begin{enumerate}[$i)$]
   \item$G$ is a generalised quaternion group of order $2^r,r\geq 3$, or
    \item$G$ is cyclic of order $p$, where $p$ is prime and $p\geq 7$, or
    \item$G$ is the abelian group $C_3\times C_3$.
\end{enumerate}
\end{corollary}

\begin{table}[htb!]
\caption{The groups $G$ mentioned in Theorem \ref{1}, $iv)$, and the values $\alpha(G)$.}
\label{table:kd}
\begin{tabular} {|c | c |c  |}
 \hline
 
  & $G$ & $\alpha(G)$ \\ 
 \hline\hline
 $1-4$ &$\;\;\;\;\;\;\; {{C}}_{12},\;\;C_{15},\;\;C_{20},\;\;C_{21}\;\;\;\;\;\;\; $& $\;\;18,\;21,\;25,\;23 \;\; $\\
 \hline 
 $5-8$ & $\;\;\;\;\;\;\;  {{C}_2}\times C_4, \;\;C_3\times C_3,\;\; C_4\times C_4,\;\; C_5\times C_5 \;\;\;\;\;\;\;$ & $12,\;18,\;20,\;30 $\\ 
 \hline 
 $9-10$ &$C _2\times C_2\times C_3,\;\;  C_2\times C_3\times C_3 \;$& $ 13,\;20 $\\
 
  \hline 
 $11-13$ & $\Dic_3,\;\; \Dic_5,\;\; \Dic_6$ &$17,\;23,\;25 $ \\ 
 \hline
 $14$ & ${G}_{16}$
 & $18$   \\ 
 \hline
 $15$ & A$_4 $& $ 16$  \\
 \hline 
 $16$ & ${G}'_{16}$ 
 &   $18$ \\ 
 \hline
 $17$ & ${Q}_{8} \times C_{3}$&   $25$ \\

 \hline

\end{tabular}

\end{table}

The main tool in the proof of Theorem \ref{1} is the GRR-Theorem (Theorem \ref{23b}). It states that, with some specified families of exceptions, every finite group $G$ has a graphical regular representation (GRR), i.e., a Cayley graph having full automorphism group isomorphic to $G$. If a group $G$ has a GRR, then $\alpha(G)\leq |G|$. Therefore, in order to prove Theorem \ref{1}, it suffices to study the exceptions in the GRR-Theorem.

Making use of the preliminary results presented in Section 2, we prove Theorem \ref{1} across Sections 3, 4, and 5.
Section 3 concerns the case of abelian groups; the key fact is that $\alpha(G)$ has been determined for every abelian group $G$, in \cite{arli}. 
Sections 4 and 5 are devoted to the non-abelian exceptional groups of the GRR-Theorem.
In Section 4, we address the non-abelian groups $G$ for which
the assertion of Theorem \ref{1} is that $\alpha(G)\leq |G|$. For these groups, we
construct a graph on at most $|G|$ vertices having automorphism group isomorphic to $G$. In Section 5,
we show that there exists no graph on at most $|G|$ vertices with automorphism group isomorphic to $G$, for the non-abelian groups $G$ for which the assertion  of Theorem \ref{1} is that $\alpha(G)>|G|$.
The values of $\alpha(Q_{2^r})$, $\alpha(Q_{2^r}\times C_2)$ and $\alpha(G)$ for the non-abelian groups $G$ in Table \ref{table:kd} are also justified in this section.

\subsection*{Acknowledgements} $\;$\\
This article is based on a dissertation that  was done as part of the course `MSc in Pure Mathematics' at Imperial College London in the academic year 2018-2019. I would like to thank the Onassis Foundation for supporting these studies with a scholarship [Scholarship ID: F ZO 023-1/ 2018-2019]. I am especially grateful to my advisor Martin W. Liebeck, who shared his enthusiasm and ideas on this topic, one of which led to Construction \ref{ip1}. I would also like to thank the anonymous referees for helpful comments.

\section{Background}
Throughout the paper, all groups and graphs mentioned are assumed to be finite.

Let us now present two families of groups that play an important role in our text.
\begin{dfn}\label{defdic}\cite{ww}
Let $A$ be an abelian group that contains an element of order $2k$ for some $k\geq 2$.
A group $G$ of the form $$G=\la A,b
\mid b^4=1,\;b^2\in A\setminus \{1\},\;bab^{-1}=a^{-1},\forall 
\;a\in A\ra$$ is called \textit{generalised dicyclic} and is denoted by $\Dic(A,b^2)$.
\vspace{1 mm}\\
If $A$ is cyclic, $G$ is simply called \textit{dicyclic}. It is denoted by $\Dic_{m}$, where $m=\frac{|G|}{4}$.
\vspace{1 mm}\\
A dicyclic group of order $2^r$ is called generalised \textit{quaternion} $({r\geq 3})$. We denote it by $Q_{2^r}.$
\end{dfn}
Note that the existence of the element of order $2k$ in Definition \ref{defdic} ensures that generalised dicyclic groups are non-abelian.

\begin{dfn}
Let $A$ be an abelian group. A group $G$ of the form $$G=\la A,b\mid b^2=1,\; bab^{-1}=a^{-1},\forall\; a\in A\ra$$ is called \textit{generalised dihedral} and is denoted by $\Dih(A)$.
\vspace{1 mm}\\
If $A$ is cyclic of order $m$, $G$ is the \textit{dihedral} group of order $2m$, which we denote by D$_{2m}$.
\end{dfn}

A \textit{graph}  $\G$ consists of a vertex set, which we denote by $V(\G)$ and an edge set, denoted by $E(\G)$; we consider an edge to be an unordered pair of vertices of $\G$. 
We denote an edge between $v,w\in V(\G)$ by $v\sim w$ or we say that $[v,w]\in E(\G)$. Moreover,
if $X$ is a subgraph of $\G$ and $v\in V(X)$, we denote by $\rho_X(v)$ the \textit{valency} of $v$ in $X$ and by $\rho(v)$ the valency of $v$ in the graph $\G$.
If a group $G$ acts on a graph $\G$ and $v\in V(\G)$, then we denote by $\ma v$ the \textit{orbit} containing $v$, $\ma v=\{gv\mid g\in G\}$, and by $G_v$ the \textit{stabilizer} of $v$, $G_v=\{g\in G\mid gv=v\}$.

Given a group $G$ and a set $S\subset G\setminus{\{\small{1}\}}$ that is inverse-closed, we define the \textit{Cayley graph} Cay$(G,S)$ to be the graph with vertex set $G$ and edges $\{x,sx\}$, for all $x \in G, s\in S.$

 A graph $\G$ is called a \textit{Graphical Regular Representation} (GRR) of a group $G$ if there exists some $S\subset G$ such that  $\te{Cay}(G,S)=\G$ and $\Aut(\G)\cong G$. The following theorem is known as the \textit{GRR-Theorem}.  It was proven by Godsil \cite{Godsil} for non-solvable groups and by Hetzel \cite{hetze} for solvable groups using previous results of several authors including \cite{im,nowit,nw,sabb,sabbb,w1,ww}.
\begin{theorem}{\cite{Godsil}}{\label{23b}} A group admits a GRR if and only if it is not an abelian group of exponent greater than 2, a generalised dicyclic group, or one of the 13 exceptional groups shown in Table \ref{table:kysymys}.
\end{theorem}
\begin{corollary}
If $G$ is a non-abelian
, non-generalised dicyclic group that is not one of the 13 groups shown in Table \ref{table:kysymys}, then $\alpha(G)\leq |G|.$
\end{corollary}

\begin{table}
\caption{The groups $G$ mentioned in the GRR-Theorem (Theorem \ref{23b}).}
\label{table:kysymys}
\begin{tabular} {|c | c | c |}
 \hline
 
  & $G$ & $|G|$\\ 
 \hline\hline
 $1-3$ &  ${C}_2\times C_2, \;\;C_2\times C_2\times C_2,\;\; C_2\times C_2\times C_2\times C_2$  & $4, 8, 16$\\
 \hline
 $4-6$ & ${D}_{6}, \;\; \it{D}_{8},\;\;\it{D}_{10}$ & $\;\;6, 8, 10\;\;$\\ 
 \hline
 $7$ & A$_4$ & $12$ \\ 
 \hline
  $8$ & $\la a,b,c\mid a^2=b^2=c^2=1,\;abc=bca=cab \ra$  & $16$\\ 
 \hline
  $9$ & ${G}'_{16}$ & $16$\\ 
 \hline
 $10$ & $\;\;\;\;\;\;\;\la a,b,c\mid a^3=b^3=c^2=1,\;ab=ba,\;(ac)^2=(bc)^2=1 \ra\;\;\;\;\;\;\;$   & $18$\\
 \hline
$11$ &  $\la a,b,c\mid a^3=c^3=1,\;ac=ca,\;bc=cb,\;b^{-1}ab=ac \ra$  & $27$\\
 \hline
$12-13$ & ${Q}_8\times C_3,\;\;\;Q_8\times C_4$ & $24, 32$\\

 \hline

\end{tabular}

\end{table}

Let us now state Babai's theorem.
\begin{theorem}{[Babai, \cite{Babai}]}\label{30} If $G$ is a group different from the cyclic groups of order 3, 4, 5 then $\alpha(G)\leq2 |G|$. 
\end{theorem}
The values of $\alpha(G)$ for cyclic groups $G$ and graph constructions can be found in \cite{arli}, which builds on the work of Sabidussi \cite{sabb}. For the non-cyclic groups, we will use the following construction, given by Babai in \cite{Babai}:
\begin{construction}{\label{con0}}
 Let $G$ be a non-cyclic group of order $|G|\geq 6$ and let $H=\{h_1,...,h_d\}$ be a minimal generating set of $G$. Let $G'$ be an isomorphic copy of $G$ with an isomorphism $g\longmapsto g'$ from $G$ to $G'$.
 We define the graphs $X_1$ and $X_3$ to be such that  
    \begin{align*}
&V(X_1)=G ,\phantom{'}\qquad E(X_1)=\big\{[gh_i,gh_{i+1}]{\mathlarger{\mid}}\;g\in G,\;i=1,...,d-1\big\}, \\
&V(X_3)=G' ,\qquad E(X_3)   =        \big\{[g'h_1',g']{\mathlarger{\mid}}\; g'\in G'\big\}.
    \end{align*}
 Let $\rho_{X_s}$ be the valency of the vertices of $X_s$, $s=1,3$. We define the graph $X_2$ to be \[X_2=\left\{\begin{array}{ll}
     X_3,\;&\;\text{     if }\rho_{X_1}\neq\rho_{X_3},\\
     \overline{X_3},\;&\;\text{     if }\rho_{X_1}=\rho_{X_3},
\end{array}\right.  \] where $\overline{X_3}$ is the complement graph of $X_3$.
     
Finally, let us define the graph $X$ such that 
 \begin{align*}&V(X)=V(X_1)\cup V(X_2), \\ &E(X)=E(X_1)\cup E(X_2)\cup\big\{\big[g',g\big],\big[g',gh_i\big]\mathlarger{\mathlarger{\mid}}\;g\in G, \;i=1,...,d\big\}.\end{align*} 
The map $g:V(X)\rightarrow V(X)$ such that
\[\;\;\;g(v)=\mathsmaller{\left\{\begin{array}{ll}
     gv,\;&\;\text{     if }v\in V(X_1),\\
     g'v,\;&\;\text{     if }v\in V(X_2),
\end{array}\right.}  \] is a graph automorphism for every $g\in G$, and $\Aut(X)\cong G$; the proof appears in \cite{Babai}.
\end{construction}

The inequality in Babai's Theorem \ref{30} does not hold for the three cyclic groups excluded.
\begin{example}{\label{opq}}
We will see shortly (Proposition \ref{27b}) that $\alpha(C_4)=10$.
A graph on 10 vertices that has automorphism group isomorphic to $C_4$ is shown in Figure 1. In particular, the automorphism group of this graph can be realised as the subgroup $\la b\ra$ of $S_{10}$, where $b=(1\; 2)(3\; 4\; 5\; 6)(7\; 8\; 9\;10)$
(\cite[Lemma 2.1.3.3.]{jess}).
\end{example}
\begin{figure}{\label{figure 1.}}
\centering
\begin{tikzpicture}
  [scale=.73,auto=left,every node/.style={circle,fill=blue!20}]
  \node (n3) at (1,10) {3};
  \node (n6) at (1,2)  {6};
  \node (n4) at (10,10)  {4};
  \node (n5) at (10,2) {5};
  \node (n7) at (3,8)  {7};
  \node (n8) at (8,8)  {8};
    \node (n10) at (3,4) {10};
  \node (n9) at (8,4)  {9};
      \node (n1) at (4.15,5) {1};
  \node (n2) at (6.85,5)  {2};

  \foreach \from/\to in {n1/n3,n1/n7,n1/n5,n1/n9,n2/n4,n2/n6,n2/n8,n2/n10,n3/n4,n3/n7,n3/n8,n3/n6,n4/n3,n4/n5,n4/n9,n4/n2,n4/n8,n5/n6,n5/n10,n5/n9,n6/n7,n6/n10}
    \draw (\from) -- (\to);
\end{tikzpicture}
\caption{}
\end{figure}

 \section{Proof of Theorem \ref{1}: abelian groups}
The aim of this section is to prove that Theorem \ref{1} holds for every abelian group $G$. 
\begin{proposition}\label{27d}
Let $G$ be an abelian group. Then one of the following holds: 
\begin{enumerate}[$i)$]
    \item $\alpha(G)\leq |G|,$ 
    \item$G$ is cyclic of order $p^k$ or $2p$ for some prime number $p$ $(|G|\neq2)$, 
    \item $G$ is one of the 10 abelian groups shown in Table \ref{table:kd}.
\end{enumerate}
  If $ii)$ or $iii)$ is true then $\alpha(G)>|G|.$
\end{proposition}
The value of $\alpha(G)$ was determined for every cyclic group $G$ by Sabidussi \cite{sabb,sabbb}, when $|G|$ is a prime number, and by Meriwether (unpublished, see \cite{sabbb}), in general. However, Arlinghaus \cite{arli} was the first to present an algorithm to compute $\alpha(G)$ when $G$ is cyclic or, more generally, abelian. Table \ref{iii} contains the value of $\alpha(G)$ for some small abelian groups, which we computed using Arlinghaus' algorithm \cite[Theorem 8.1]{arli}.

\begin{proposition}\label{27}\cite[Theorem 8.1]{arli}
Consider the abelian group $G= C_{q_1}\times C_{q_2}\times\cdots \times C_{q_s}$, where $q_i$ is a prime power, $i=1,...,s$. Then,
\begin{equation}\label{10.11}\alpha(G)\leq \alpha(C_{q_1})+\alpha(C_{q_2})+\cdots +\alpha(C_{q_s}),\end{equation}
\vspace{-7 mm}
\begin{equation}\label{popa}
    \alpha(C_{2}\times C_{q_2})=2+\alpha(C_{q_2}).\;\;\end{equation}
\end{proposition}

\begin{proposition}\label{27b}\cite[Theorem 5.4]{arli}
Let $p$ be a prime number and $r$ be a positive integer. Then \vspace{-1.5 mm}
\[
\alpha(C_{p^r})=\left\{ 
\begin{array}{ll}
     2,\;&\;\text{if }p^r=2,\\
     p^r+2p, \;&\; \text{if }p=3,5, \\
     p^r+6, \;&\;\text{if }p=2,\;r\geq 2,\\
     p^r+p, \;&\;\text{if }p\geq7.
     
\end{array} 
\right.  \thinspace \thinspace \thinspace \thinspace \thinspace \thinspace \thinspace \thinspace \thinspace \thinspace \thinspace \thinspace \thinspace \thinspace \thinspace \thinspace \thinspace \thinspace \thinspace \thinspace 
\]

\end{proposition}

\begin{table}
    \caption{The values of $\alpha(G)$ for certain abelian groups $G$.}
    \label{iii}
   \hspace{-10 mm} \begin{minipage}{.5\linewidth}
$$G= C_{p_1^{r_1}}\times C_{p_2^{r_2}} \;\;(p_1^{r_1}{\small \leq{p_2^{r_2}})}$$
      \vspace{3 mm}
      \centering
\begin{tabular} {|c||c | c |c  |c | c |c |}
 \hline
 
\footnotesize{\diagbox[width=1.25cm, height=1.3cm]{{$p_1^{r_1}$}}{\raisebox{4.9pt}{ $p_2^{r_2}$}}}& \thinspace 2\thinspace&3&4&5&7&8 \\ 
 \hline\hline
 2&4&11&12&17&16&16  \\
 \hline 
 3& &18&18&21&23&22 \\ 
 \hline 
 4 & & & 20&25&24&24  \\ 
 \hline
5 & & & &30&29&29 \\ 
 \hline

\end{tabular}
    \end{minipage}%
    \begin{minipage}{.5\linewidth}
      \centering
        $$\;\;\;\;\;\;\;\;\;\;\;\;\;\;\;\;\;G= C_2\times C_{p_2^{r_2}}\times C_{p_3^{r_3}}  \;\;(p_2^{r_2}{\small \leq{p_3^{r_3}})\;}$$
\begin{tabular} {|c||c | c |c  |c | c |c |c | c |c |}
 \hline
 
\footnotesize{\diagbox[width=1.25cm, height=1.3cm]{{$p_2^{r_2}$}}{\raisebox{4.9pt}{ $p_3^{r_3}$}}}& \thinspace 2\thinspace&3&4&5&7&8&9&11&13 \\ 
 \hline\hline
 2&6&13&14&19&18&18&19&26&30  \\
 \hline 
 3& &20&20&23&25&24& 23&33&37 \\ 
 \hline 
 4 & & & 22&27&26&26&26&34&38  \\ 
 \hline

\end{tabular}
\vspace{9.5 mm}
    \end{minipage} 
\end{table}

Proposition \ref{27b} gives rise to the following inequalities that are essential for the proof of Proposition \ref{27d}: 
 \begin{equation}\label{ararar}
   \alpha(C_{p^r})\leq 3p^r,\;\;\;\;\;\;\;\;\;\; \;\;\;\;\;\;\;\;\;\;\; \;\;\;\;\;\;\;\;\;\end{equation}
   \vspace{-6 mm}
   \begin{equation}\label{arara}
\alpha(C_{p^r})\leq 2p^r,  \;\;\; \text{ if } p^r\geq 7,\;\;\;\;\;\;\;\;\;\;\; 
    \end{equation}
       \vspace{-6 mm}
    \begin{equation}\label{ararat}
  \alpha(C_{p^r})\leq p^r+\te{max}\big\{6, 2p\big\}.  \;\;\;\;   \;\;\;\;\;
    \end{equation}

 In preparation for proving Proposition \ref{27d}, we establish the following lemma.
\begin{lemma}\label{09e} Proposition \ref{27d} holds when 
$G$ is a direct product of two cyclic groups of prime-power order.
\end{lemma}
\begin{proof} Let $G=C_{p_1^{r_1}}\times C_{p_2^{r_2}}$ for some prime powers $p_1^{r_1}, p_2^{r_2}$ such that $p_1^{r_1}\leq p_2^{r_2}$.
Using Table \ref{iii} we deduce that the Lemma holds when $p_1^{r_1}\leq 5$ and $  p_2^{r_2}\leq 8$. If $p_1^{r_1}>5 $ then the inequalities (\ref{10.11}) and (\ref{arara}) imply that
\begin{equation*}
\alpha(G) \leq\alpha(C_{p_1^{r_1}})+\alpha(C_{p_2^{r_2}})\leq 2p_1^{r_1}+2p_2^{r_2}\leq 4p_2^{r_2}<|G|.\end{equation*}
Hence we make the assumption that $p_1^{r_1}\leq 5$ and $p_2^{r_2}>8$.

If $|G|=2p_2$ then (\ref{popa}) and Proposition \ref{27b} imply that \hspace{1 mm}$\alpha(C_{2p_2})= 2+2p_2$,\hspace{1 mm} thus $\alpha(G)> |G|$. 
On the other hand, if $|G|=2p_2^{r_2}$ and $r_2>1$, then the inequality
\begin{equation}\label{nki} 2+\te{max}\big\{6, 2p_2\big\}< p_2^{r_2}\end{equation}
holds;
indeed, we assumed that $p_2^{r_2}>8$, so (\ref{nki}) holds in case $p_2=2$; if $p_2\geq 3$ then $2+2p_2<3p_2\leq p_2^{r_2}$.
It follows from (\ref{popa}),
(\ref{ararat}) and (\ref{nki}) that 
\[\alpha(C_{2p_2^{r_2}})\leq 2+p_2^{r_2}+\te{max}\big\{6, 2p_2\big\}< 2p_2^{r_2}=|G|.\]

If $p_1^{r_1}=3$ then by (\ref{10.11}), (\ref{arara}) and Proposition \ref{27b}, we get that $\alpha(G)\leq 9+2p_2^{r_2}\leq 3p_2^{r_2}=|G|$.

Finally, if $4\leq p_1^{r_1}\leq 5$ then it is implied by (\ref{10.11}), (\ref{arara}) and Proposition \ref{27b} that
\[\vspace{-8 mm}\alpha(G)\leq 15+ 2p_2^{r_2}< 4p_2^{r_2}\leq |G|.\] \end{proof}

{\textit{Proof of Proposition \ref{27d}}.} If $|G|=1$ then $\alpha(G)=|G|$.
 Let $G=C_{p_1^{r_1}}\times C_{p_2^{r_2}}\times \cdots\times C_{p_s^{r_s}}$, where $p_1^{r_1}\leq \cdots \leq p_s^{r_s}$ are prime powers.

If $s=1$ or $s=2$ then the statements in Proposition \ref{27d} hold for $G$ as a consequence of Proposition \ref{27b} or
Lemma \ref{09e}, respectively.

Let $s=3$. If $p_1^{r_1}p_2^{r_2}\geq 9$, then using inequalities (\ref{10.11}) and (\ref{ararar}) we conclude that \[\alpha(G)\leq 3p_1^{r_1}+3p_2^{r_2}+3p_3^{r_3}\leq 9p_3^{r_3}\leq |G|.\]
Assume now that $p_1^{r_1}p_2^{r_2}< 9;$ thus, $p_1^{r_1}=2$ and $2\leq p_2^{r_2}\leq 4$. If $p_3^{r_3}<16$ then using Table \ref{iii} we verify that the claim in Proposition \ref{27d} holds for $G$. If $p_3^{r_3}\geq 16$ instead, then (\ref{10.11}) and (\ref{ararar}) together with Proposition \ref{27b} imply that
\[\alpha(G)\leq 2+3p_2^{r_2}+3p_3^{r_3}< 4p_3^{r_3}\leq |G|.\]

Let $s=4.$ If $p_1^{r_1}p_2^{r_2}p_3^{r_3}\geq 12$, then by (\ref{10.11}) and (\ref{ararar}) we have \[\alpha(G)\leq 3p_1^{r_1}+3p_2^{r_2}+3p_3^{r_3}+3p_4^{r_4}\leq 12p_4^{r_4}\leq |G|.\] Otherwise $p_1^{r_1}=p_2^{r_2}=p_3^{r_3}=2$, in which case (\ref{10.11}), (\ref{ararar}) and Proposition \ref{27b} show that \vspace{-3 mm}
\[\alpha(G)\leq 2+2+2+3p_4^{r_4}< 8p_4^{r_4}= |G|.\]

Finally, let us assume that $s\geq 5.$ Then, using (\ref{10.11}) and (\ref{ararar})
we conclude that
\[\alpha(G)\leq3p_1^{r_1}+3p_2^{r_2}+\cdots +3p_s^{r_s}\leq 3s{p_s}^{r_s}.\]
Furthermore, since $3s<2^{s-1}$ and $p_i^{r_i}\geq 2$ for every $i\in \{1,...,s-1\}$, we have that
\[3s{p_s}^{r_s}< 2^{s-1}{p_s}^{r_s}\leq {p_1}^{r_1}{p_2}^{r_2}\cdots{p_s}^{r_s}=|G|.\] Hence $\alpha(G)<|G|.$ \vspace{-8 mm}
\begin{flushright}
$\square$
\end{flushright}

\section{Proof of Theorem \ref{1}: the bound $\alpha(G)\leq |G|$}

In this section we prove the bound $\alpha(G)\leq |G|$ for groups $G$ that are non-abelian and do not satisfy $iii),iv)$ in Theorem \ref{1}, as summarised in the following theorem.

\begin{theorem}{\label{3}}
Let $G$ be a non-abelian group such that
\begin{enumerate}[$i)$]
     \item$G$ is not a generalised quaternion group,
     \item $G$ is not a generalised dicyclic group of the form $Q_{2^r}\times C_2$,
    \item$G$ is not one of the groups shown in Table \ref{table:kd}.
\end{enumerate}
Then $\alpha(G)\leq |G|$.

\end{theorem}

By the GRR-Theorem (Theorem \ref{23b}), in order to prove Theorem \ref{3}, we only need to consider the cases when
$G$ is generalised dicyclic and when $G$ is one of the non-abelian groups that appear in Table \ref{table:kysymys} but not in Table \ref{table:kd}. We will do this in 
Propositions \ref{3aa} and \ref{3a}.
In particular, in Proposition \ref{3aa}, we consider the groups $D_6,\;D_8,\; D_{10}$, and
\begin{equation}\label{eq:longeq}
\begin{split}
G_1=&\;\la a,b,c\mid a^2=b^2=c^2=1,abc=bca=cab \ra,
     \\G_2=&\;\la a,b,c\mid a^3=b^3=c^2=1,ab=ba,(ac)^2=(bc)^2=1 \ra ,
    \\G_3=&\;\la a,b,c\mid a^3=c^3=1,ac=ca,bc=cb,b^{-1}ab=ac\ra,
    \\G_4=&\; Q_8\times C_4 ,
    \\G_{r+2}=&\;Q_{2^{r}}\times C_2\times C_2\times C_2, \qquad r\geq 3;
\end{split}    
\end{equation}
the remaining groups are addressed in Proposition \ref{3a}.

Let us start with two lemmas that will be used in the proof of Proposition \ref{3a}.
\begin{lemma}{\label{a33}}
Let $G=\Dih(X)$ be a generalised dihedral group of order $2k$, where $k\geq 6,\;k\neq 9$, that is not the group $C_2\times C_2 \times C_2\times C_2$. Then there exists a GRR for $G$.
\end{lemma}

\begin{proof}
By the GRR-Theorem, it suffices to prove that $G$ is non-generalised dicyclic and not one
of the groups appearing in Table \ref{table:kysymys}.

Let $G=\la X,b\ra$, $ b^2=1$.
Suppose that $G=\Dic(A,c^2),$  for some $A\leqslant G,$ $c\in G$.
Then the order of $c$ is 4 and the order of $b$ is 2, hence $c\in X$, $b\in A$. It follows from the properties of generalised dicyclic and generalised dihedral groups that
$$bcb^{-1}=c^{-1}\qquad \te{and}\qquad cbc^{-1}=b^{-1}.$$
The equalities given above imply that $c^2=1$, which is a contradiction.

The restriction $k\neq 9$, implies that $G$ is not the group of order 18 in Table \ref{table:kysymys}. Moreover, since $|G|=2k\geq 12,$ $G$ is not among the groups $C_2\times C_2$, $C_2\times C_2\times C_2$, $D_6$, $D_8$, $D_{10}$ or the group of order 27 in  Table \ref{table:kysymys}. On the other hand, the group $A_4$ has no abelian subgroup of index 2, hence it is not generalised dihedral. The remaining 4 suitable groups given in Table \ref{table:kysymys} contain a central element of order 3 or 4. However, if $g$ is in the center of $G$ and $G$ is non-abelian then $g\in X$, hence $bgb=g^{-1}$. Furthermore, $bgb=g$, as $g$ is central. Therefore, $g$ has order 2.
\end{proof}
\begin{lemma}\label{30a}
Let $G$ be an abelian $2$-group and let $c\in G$ be an element of order $2$. Then there exists some $y\in G,\;A<G$ such that 
\[ G= \la y \ra\oplus A \;\;\;\te{and}\;\;\;c\in \la y \ra .\]
\end{lemma}
\begin{proof} We may choose $a_1,a_2,...,a_n,$ such that
$G=\la a_1\ra\oplus \la a_2\ra\oplus \cdots \oplus \la a_n\ra$ and
$c=\big(a_1^{2^{r_1-1}},\;a_2^{2^{r_2-1}},...,\;a_s^{2^{r_s-1}},1,1,...,1\big),\;\te{where}\;0<r_1\leq r_2\leq\cdots \leq r_s\;\te{and } a_i\te{ has order }2^{r_i}.$

Let $y=\big(a_1,\;a_2^{2^{r_2-r_1}},...,\;a_s^{2^{r_s-r_1}},1,1,...,1\big). $
Then $c\in \la y\ra$. Furthermore, the order of $y$ is equal to the order of $a_1$, which is $2^{r_1}$.
Hence $\la y\ra\cap\big( \la a_2\ra \oplus \cdots \oplus \la a_n\ra\big)=\{1\}$ and 
 $G=\la y\ra\oplus A,$
where $A=\la a_2\ra \oplus \cdots \oplus \la a_n\ra.$
\end{proof}

Let us now define a collection of graphs, one for each group appearing in (\ref{eq:longeq}).
\begin{construction}\label{4grrrrr}
Let us first define the graph $\G_1$ on 16 vertices and 52 edges. Let
$V(\G_1)=V_1\cup V_2,$ where $V_1=\{1,2,...,8\}$ and $V_2=\{1',2',...,8'\}$. Let $E(\G_1)$ be such that, for $v,w\in V_1$,
\vspace{-2 mm}\begin{align*}
v\sim w &\iff(v,w) \in\bigcup_{i,j\;\in\{0,1\}}\big(\{1+i+4j,3+i+4j\}\times \{5+i-4j,7+i-4j\}\big) ;\\
 v'\sim w'& \iff v,w\in\bigcup_{i\in\{0,1\}}\big\{1+4i,2+4i,3+4i,4+4i\big\},\qquad v\neq w; \\
v\sim w'& \iff \mathsmaller{\left\{\begin{array}{ll}
w-v=0,4, \qquad\te{or}\\
v-w\equiv 2\;(\te{mod  }{4}),v>4,w\leq4, \qquad\te{or}\\
v-w\equiv \pm 1\;(\te{mod }{4})\te{ and }(v>4\iff w>4).
\end{array}\right.}
\end{align*}

\end{construction}

\begin{construction}{\label{con3}}\label{4grrrr}
Let us now define the graph $\G_2$, which has 18 vertices and 99 edges. Let $V(\G_2)=W_1\cup W_2,$ where $W_1=\{1,2,...,9\},$ $W_2=\{1',2',...,9'\}$, and
$E(\G_2)$ is such that,
for $v,w$ $\in W_1$,
\vspace{-2 mm}\begin{align*}v'\sim w'&\iff \mathsmaller{\left\{\begin{array}{ll}
\forall k\in \{0,1,2\},v>3k\iff w>3k, \qquad v\neq w,\qquad\te{or}\\
    w-v\equiv \pm\; 3\; \te{(mod }9), \qquad \te{or}\\
    {[v,w]}\in\big\{[1, 6],[2, 4],[3, 5]\big\};
\end{array}\right.}  \\
v\sim w&\iff v'\nsim w',\qquad v\neq w;\\
v\nsim w'& \iff v,w\in \{1,2,3\},\qquad\te{or }v,w\in \{4,5,6\}.\end{align*}
\end{construction}

\begin{construction}\label{4grrr}
Let us also construct the graph $\G_3$ on 27 vertices and 171 edges. Let $V(\G_3)=\ma 1\cup \ma {1'}\cup \ma {1''},$ where $\ma 1=\{1,2,...,9\},$ $\ma {1'}=\{1',2',...,9'\}$ and $\ma {1''}=\{1'',2'',...,9''\}$. The edge set of $\G_3$ is such that, for every $ v,w\in \ma 1$, we have 
\vspace{-2 mm}\begin{align*}
v'\nsim w';& \;v''\nsim w'';\\
v\sim w &\iff v\neq w;\\
v\sim w''&\iff \mathsmaller{\left\{\begin{array}{ll}
w-v\equiv 1\;(\te{mod 3}),\qquad\te{or}\\
w-v\equiv 0,3\;(\te{mod 9});
\end{array}\right.}\\
v\sim w'&\iff  \mathsmaller{\left\{\begin{array}{ll}
w-v\equiv 0,2k,4k\;(\te{mod 9}),v\equiv k\;(\te{mod 3})\te{ for some }k\in\{1,2\}, \qquad\te{ or}\\
w-v\equiv 0,\pm 1\;(\te{mod 9}),
v\equiv 0\;(\te{mod 3});
\end{array}\right.}  
\\v'\sim w''&\iff w\sim v'.\end{align*}
\end{construction}
\begin{construction}\label{4grr}
Let $\G$ be the graph on 10 vertices with automorphism group Aut$(\G)\cong C_4$ given in Example \ref{opq}.
Let $\G'$ be a graph on 16 vertices constructed according to Babai's Construction \ref{con0} for the group $ Q_8$. 
We define $\G_4$ to be the graph $\G_4=\G\cup \G'$. 
\end{construction}
\begin{construction}{\label{4gr}}
Let $r\geq 3$. We let $\G$ be a graph on $2^{r+1}$ vertices constructed according to Babai's Construction \ref{con0} for the generalised quaternion group $ Q_{2^{r}}$ and $\G'$ to be a graph on 6 vertices such that $\Aut(\G')\cong C_2\times C_2\times C_2$, which exists since $\alpha(C_2\times C_2\times C_2)=6$ (see 
Table \ref{iii} and \cite[Theorem 8.1]{arli} for a proof).
Then, the graph $\G_{r+2}=\G\cup \G'$ on $2^{r+1}+6$ vertices has automorphism group $\Aut(\G_{r+2})\cong G_{r+2}$, since $\G$
 is a connected component of $\G_{r+2}$ of size $2^{r+1}> 6$ and
$\Aut(\G)\cong  Q_{2^{r}},\;\Aut(\G')\cong C_2\times C_2\times C_2$.
\end{construction}

\begin{proposition}{\label{3aa}}
If $G$ is the dihedral group $D_{2n}$, where $3\leq n\leq 5,$ or one of the groups $G_1, G_2, G_3, G_4, G_{r+2}$ $(r\geq 3)$ in (\ref{eq:longeq}) then $\alpha(G)\leq |G|$.
\end{proposition} 
\begin{proof}
The $n$-cycle has full automorphism group $D_{2n}$ so $\alpha(D_{2n})< |D_{2n}|$ for $3\leq n\leq 5$.

The graphs $\G_i$ in Constructions \ref{4grrrrr}$-$\ref{4gr} 
are designed to have at most $|G_i|$ vertices and automorphism groups $\Aut(\G_i)\cong G_i$,
for each $i\geq 1$. We omit the proof that $\Aut(\G_i)\cong G_i$ for $1\leq i\leq 4$, which we verified using the mathematical software GAP \cite{gap}.
\end{proof}

We will now show that Theorem \ref{3} also holds for the generalised dicyclic groups $G$ that are different from $Q_{2^r}\times C_2\times C_2\times C_2, r\geq 3$, completing the proof of Theorem \ref{3}.
\begin{proposition}{\label{3a}}
Let $G$ be a generalised dicyclic group such that
\begin{enumerate}[$i)$]
     \item$G$ is not a generalised quaternion group,
     \item $G$ is not a generalised dicyclic group of the form $Q_{2^r}\times C_2$ or $Q_{2^r}\times C_2\times C_2\times C_2$,
    \item$G$ is not one of the groups $\Dic_3,\; \Dic_5,\;\Dic_6,\;{G}_{16}$ that appear in Table \ref{table:kd}.

\end{enumerate}
Then $\alpha(G)\leq |G|$.

\end{proposition}

The rest of this section concerns the proof of Proposition \ref{3a}.

 Let $G=\Dic(A,b^2)$ be a generalised dicyclic group as in Proposition \ref{3a} and let \[A=A_2\oplus A_{2'},\] where $A_2$ is the Sylow 2-subgroup of $A$
 and $A_{2'}$ is the Hall $2'$-subgroup of $A$.
Then, by Lemma \ref{30a}, there exist $y\in A_2$ and $B_2<A_2$ such that \vspace{-1 mm} \[A_2= \la y \ra \oplus B_2 \;\;\;\;\te{and}\;\;\; b^2\in \la y \ra.\] 
Setting $X=B_2\oplus A_{2'}$, we get $A=X\oplus \la y \ra.$

\vspace {1.5 mm} Let $r,k$ be such that $\la y \ra\cong C_{2^r}, |X|=k$. We note that the quotient group $G/X$ is isomorphic to the generalised quaternion group $Q_{2^{r+1}}$, for $r>1$, and to the cyclic group $ C_4$, for $r=1$. 
 Moreover, the quotient group $G/{\la y\ra}$ is isomorphic to the generalised dihedral group $\Dih(X)$.

\begin{construction}{\label{ip1}}
 We will construct a graph $\G$ such that $\Aut(\G)\cong G.$ We start by defining two graphs, $\G_1$, $\G_2$, with the property that $\Aut(\G_1)\cong G/X$, $\Aut(\G_2)\cong G/\la y\ra$.

For $r\geq 2$,
we let $\G_1$ be the graph with vertex set $V(\G_1)=G/X\cup (G/X)'$ that arises from Babai's Construction \ref{con0} for the generalised quaternion group $G/X$ with respect to the minimal generating set $H=\{yX,bX \}$.
Furthermore, we partition the set of vertices $G/X$ of $\G_1$ into the sets $T_1,\;T_2$, where $ T_i=\{y^nb^{(i-1)}X\mid n\in\mathbb{N}\}, i=1,2$. 
For $r=1$, we define $\G_1$ to be the graph with automorphism group isomorphic to the cyclic group $C_4$ that was presented in Example \ref{opq}.
Likewise, we partition its vertex set into the sets $T_i$, where $T_i=\{2k+i\mid 0\leq k\leq 4\}$,
$i=1,2.$

 Let us describe the graph $\G_2$ for all values of $k$. The conditions $i),ii)$ in Proposition \ref{3a} ensure that $k\geq 3$.
 For $3\leq k\leq 5$, we construct $\G_2$ with vertex set $V(\G_2)=G/\la y\ra\cup (G/\la y\ra)'$ according to Babai's Construction \ref{con0} for the generalised dihedral group $G/\la y\ra$, with respect to some minimal generating set $K=\{k_1,k_2,...,k_d \}$ of $G/\la y\ra$ such that $k_1=b\la y\ra$.
For $k\geq 6,\;k\neq 9$,  we choose a GRR for $G/\la y\ra$, which exists by assumption $ii)$ in Proposition \ref{3a} and Lemma \ref{a33}, and define the graph $\G_2$ to be either this GRR or its complement, with the additional property that the number
$k+\rho_{\G_2}(v)$ is even, for $v\in V(\G_2)$.
Furthermore, for $k\geq3,\;k\neq 9$,
we partition the set of vertices $ G/\la y\ra$ of $\G_2$ into $S_i=\{xb^{(i-1)}\la y\ra\mid x\in X\}$, $i=1,2$.
Finally, for $k=9$, we let $\G_2$ be the graph on 18 vertices presented in Proposition \ref{3aa}, and $S_i=W_i$, where $W_i$ is as in Construction \ref{con3}, for $i=1,2$.

Let us now define the graph $\Gamma$ such that \begin{align*}
  &  V(\Gamma)=V(\Gamma_1)\cup V(\Gamma_2); \\& E(\Gamma)=E(\Gamma_1)\cup E(\Gamma_2)\cup E,\qquad\text{where} \\& E=\Big\{\big[t_i,s_i\big] \mid t_i\in T_i,\;s_i\in S_i,\;i=1,2\Big\}.
\end{align*}

\end{construction}

\begin{lemma}
Let
the sets of vertices $V(\G_1),\;V(\G_2)$ of $\G$ be fixed by $\phi$, for every $\phi\in \Aut(\G)$. Then $\Aut(\G)\cong G.$
\end{lemma}

\begin{proof}
By construction, the groups $G/X$ and $G/\la y\ra$ act on $V(\G_1)$ and $V(\G_2)$, respectively. Using these actions, we associate a map $g:V(\G)\rightarrow V(\G)$ to every $g\in G$ by setting

\[g(v)=\begin{cases}
gX(v),& \text{if $v\in V(\G_1)$}\\
g\la y \ra(v),& \text{if $v\in V(\G_2)$}.
\end{cases}\]
In other words, the map $g$ is defined to satisfy $g\restriction_{V(\Gamma_1)}=gX$ and $g\restriction_{V(\Gamma_2)}=g\langle y \rangle$.

Let $g\in G$.  We will show that the map $g$ is an automorphism of $\G$.
If $v,w\in V(\G_i)$ for some $i\in \{1,2\}$ then $g$ is an automorphism of $\G$, since 
\[v\sim w \;\te{ in } \G\iff v\sim w\; \te{ in } \G_i\iff g\restriction_{V(\G_i)}v\sim g\restriction_{V(\G_i)}w \;\te{ in } \G_i\iff gv\sim gw\;\te{ in } \G. \]
Let $v\in V(\G_1), w\in V(\G_2)$. 
By construction of $\G$, $g(T_i\times S_i)=T_j\times S_j,$ where i=j if and only if $g\in \la X,y\ra,$ $i,j\in \{1,2\}$. Thus,
\[ [v, w]\in E(\G) \iff (v,w)\in \bigcup_{i=1}^{2}(T_i\times S_i) \iff (gv,gw)\in \bigcup_{i=1}^{2}(T_i\times S_i). \]
In other words, $v\sim w \iff gv\sim gw$.
Since $\la y\ra\cap X=\{1\}$, we have $G\leq \Aut(\G)$. 

For the opposite inclusion,
let $\phi \in \Aut(\G)$. Since, by assumption, $\phi$ fixes $V(\G_1)$, the restriction $\phi\restriction_{V(\G_1)}$ of $\phi$ is an automorphism of $\G_1$. Since $\Aut(\G_1)\cong G/X$, we have that $\phi\restriction_{V(\G_1)}= g_1 X$, for some $g_1\in  G$. Then the automorphism
 $g_1^{-1}\phi$ acts trivially on $V(\G_1)$.
As the set of vertices $T_1$ is fixed by $g_1^{-1}\phi$, the set consisting of all neighbours of $T_1$ is also fixed by the same automorphism.
 However, $g_1^{-1}\phi$ fixes all neighbours of $T_1$, except, possibly, from elements of the set $S_1$. Therefore, $S_1$ is fixed by $g_1^{-1}\phi$.

Likewise, we consider the restriction of the automorphism $g_1^{-1}\phi$ on $\G_2$ to conclude that there exists some $g_2\in G $ such that the automorphism $g_2^{-1}g_1^{-1}\phi$ acts trivially on $V(\G_2)$; without loss of generality, we let $g_2\in \la X,b\ra$.
Since the graph automorphisms $g_1^{-1}\phi$ and $g_2^{-1}g_1^{-1}\phi$ fix the set $S_1$, we conclude that $g_2\in X$.
However, $X$ acts trivially on $V(\G_1)$. Therefore, the graph automorphism $g_2^{-1}g_1^{-1}\phi$ is the identity automorphism of $\G$. Thus, $$\phi=g_1 g_2 \in G $$ and hence $\Aut(\G)\leq G.$
\end{proof}

\begin{lemma}
Let $\phi\in \Aut(\G)$. The sets of vertices $V(\G_1),\;V(\G_2)$ are fixed by $\phi$.
\end{lemma}

\begin{proof}
First, we partition $V(\G_1)$, $V(\G_2)$ into the sets $V(X_i),\; V(Y_i)$, where $i\in\{1,2\}$, and 
\[ 
\;V(X_1)= \left\{
\begin{array}{ll}
      G/X,  &\; \te{if }  \;r\geq2\\
      V(\G_1), &\; \te{if }\;r=1, 
\end{array} 
\right. 
\;\;\;\;\;\;\;\;\;\;\;\;\;\;\;
\;V(X_2)= \left\{
\begin{array}{ll}
      (G/X)',  &\; \te{if }  \;r\geq2\\
      \emptyset, &\; \te{if }\;r=1, 
\end{array} 
\right. 
\]
\[\;\;\;\;\;\;\;
\;\;V(Y_1)= \left\{
\begin{array}{ll}
      G/\la y \ra,  &\; \te{if }  \;k\neq 9\\
      V(\G_2), &\; \te{if }\;k=9, 
\end{array} 
\right. 
\;\;\;\;\;\;\;\;\;\;\;\;\;\;\;
\;V(Y_2)= \left\{
\begin{array}{ll}
      (G/\la y \ra)', &\; \te{if }\;3\leq k \leq 5 \\
      \emptyset,  &\; \te{if }  \;k\geq 6.
\end{array} 
\right. 
\]
We will examine all possible values of $r,k$ in order to show that $V(\G_2)$ is fixed by $\phi$, using the automorphisms' property to preserve the valency of the vertices permuting.

The valency of a vertex $v\in V(\G_i)$ in $\G$ is
\begin{equation}\label{15.}
\rho(v)=\rho_{\G_i}(v)+\nu_{\G_j}(v),\end{equation} where $\nu_{\G_j}(v)$ is the number of neighbours of $v$ that lie in $V(\G_j)$ and $i,j\in\{1,2\}, i\neq j$.
\medskip\\
\textbf{Case 1:} Assume that $r\geq 2,\;k\geq 6$. Let $x_i\in V(X_i),\;y_i\in V(Y_i),\;i\in\{1,2\}$. Then, by (\ref{15.}), 
\begin{align*}
& \rho(x_1)=5+k, &
\rho(y_1)=\rho_{\G_2}(y_1)+2^r,   
\end{align*} which, together with the assumption that the number $k+ \rho_{\G_2}(y_1)$ is even, implies that $\rho(x_1)\neq \rho(y_1)$.

Suppose now that $\phi(v)\in V(\G_1)$, for some $v\in V(\G_2)$. Then $\rho(v)= \rho(y_1)\neq \rho(x_1)$, hence $\phi(v)\in V(X_2)$. Since $\rho(\phi(v))=\rho(x_2)\neq \rho(x_1)$, the set $V(X_1)$ is fixed by $\phi.$ Therefore, the number of neighbours of $v$ that lie in $V(X_1)$, which is $2^r$, is equal to the number of neighbours of $\phi(v)$ in $V(X_1)$, which is 3; a contradiction. Hence $V(\G_2)$ is fixed by $\phi$.
\medskip\\
\textbf{Case 2:} Suppose that $3\leq k\leq 5$; we assume that $G$ satisfies Proposition \ref{3a}, $iii)$, thus $r\geq 2$.
If $x_i\in V(X_i),\;y_i\in V(Y_i),\;i\in\{1,2\}$, then, by (\ref{15.}),
\begin{align*}
&\rho(x_1)=5+k, & \rho(y_1)= \rho_{\G_2}(y_1)+2^r, &
\\ &\rho(x_2)= 2^{r+1}, & \rho(y_2)= d+2 . \;\;\;\;\; \;\;\;\;\;&
\end{align*}
We will show that the sets of vertices $V(Y_2)$ and $V(Y_1)$ are fixed by $\phi$.
Considering all abelian groups of order $3,4$ or $5$, we conclude that the size of a minimal generating set of a generalised dihedral group of size $2k$, such that $3\leq k \leq 5$, is between 2 and 3; hence $4\leq \rho(y_2)\leq 5$.
On the other hand, by construction, min$\{\rho(x_1),\rho(x_2),\rho(y_1)\}>5 $. As graph automorphisms preserve the valency of the vertices they permute, $V(Y_2)$ is fixed by $\phi.$ It is implied that the set of neighbours of $V(Y_2)$, which is $V(Y_1)$, is also fixed by $\phi$.
\medskip\\
\textbf{Case 3:} Assume now that $r=1,\;k\geq 6,\;k\neq 9$.
First, we will compute the valency of each vertex of $\G$. By (\ref{15.}),
the valency of the vertices that lie in the sets of vertices $\{1,2\}$, $\{3,4,5,6\}$, $\{7,8,9,10\}$ is $k+4$, $k+5$, $k+3$, respectively, 
and the valency of the vertices $v\in V(\G_2)$ is $\rho _{\G_2}(v)+5$. 

Suppose that $\phi(v)\in V(\G_1)$ for some $v\in V(\G_2).$ Without loss of generality, we assume that $v\in S_1$. Graph automorphisms preserve the valency of the vertices they permute so \begin{align*}& \rho_{\G_2}(v)+2=k+j,& \text{for some $j\in\{0,1,2\}.$}\end{align*}
The graph $\G$ was constructed so that the number $\rho_{\G_2}(y_1)+ k$ is even, hence $j$ is even. In other words, the set of vertices $\{1,2\}$ is fixed by $\phi$, as is either the set $\{7,8,9,10\}$ or the set $\{3,4,5,6\}$. The vertex $1\in V(\G_1)$ is connected to all other vertices in $T_1=\{1,3,5,7,9\}$ and no vertex in $T_2$. Furthermore, the set of neighbours of $v$ that lie in $V(\G_1)$ is $T_1$. These properties combine to say that $\phi(v)$ is adjacent to either the vertices $1,3+2j,5+2j$ or the vertices $2,4+2j,6+2j$. However, there exists no vertex in $V(\G_1)$ that is adjacent to any of these triplets of vertices. Thus, $\phi$ fixes $V(\G_2)$. 
\medskip\\
\textbf{Case 4:}
Finally, let $r=1,\;k=9.$ We confirmed that the graph $\G$ on 28 vertices constructed has the desired property using the mathematical software GAP \cite{gap}.
\end{proof}
The last step in the proof of Proposition \ref{3a} is to show that the order of the graph $\G$ constructed is bounded by the order of the group $G$.
\begin{lemma}
The graph $\G$, defined in Construction \ref{ip1}, has at most $|G|$ vertices.
\end{lemma}
\begin{proof} The graph $\G$ was constructed so that the set $V(\G_1) $ has size 10, for $r=1$, and $2|G/X|$, for $r\geq 2$. Moreover, the size of $V(\G_2) $ is $2|G/\la y\ra|$, for $3\leq k\leq 5$, and $|G/\la y\ra|$, for $k\geq 6.$ Thus,
\[|V(\G)|=|V(\G_1)|+|V(\G_2)|={\te{max}}\Big\{2\;\frac{|G|}{|X|},\;10\Big\}+\Big(1+\Big\lfloor\frac{5}{k}\Big\rfloor\Big)\;\frac{|G|}{|\la y \ra|},\]
where $\lfloor\frac{5}{k}\rfloor $ is the integer part of the real number $\frac{5}{k}$. Let us now explain why $|V(\G)|\leq |G|$.\\
If $k=3$ then the assumption that $r\geq 3$
\big(Proposition \ref{3a}, $iii)$\big)
implies that  $\;|V(\G)|= \frac{2}{3}|G|+\frac{1}{2^{r-1}}|G|<|G|.$ 
Similarly, if $4\leq k\leq 5$ then $r\geq2$, 
hence $|V(\G)|\leq \frac{1}{2}|G|+\frac{1}{2^{r-1}}|G|\leq|G|.$
Finally, if $k\geq 6$ then $|V(\G)|\leq \frac{1}{3}|G|+\frac{1}{2^{r}}|G|<|G|.$
\end{proof}

\section{Proof of Theorem \ref{1}: the bound $\alpha(G)>|G|$}

In this section we prove the bound $\alpha(G)> |G|$ and compute $\alpha(G)$ for groups $G$ that are non-abelian and satisfy one of $iii),iv)$ in Theorem \ref{1}.

\begin{theorem}{\label{203}}
Let $G$ be a group such that one of the following holds:
\begin{enumerate}[$i)$]
     \item$G$ is a generalised quaternion group,
     \item $G$ is a generalised dicyclic group of the form $Q_{2^r}\times C_2$,
    \item$G$ is one of the non-abelian groups that appear in Table \ref{table:kd}.
\end{enumerate}
Then $\alpha(G)> |G|$;
indeed, if $i)$ holds, $\alpha(G)=2|G|$; if $ii)$ holds, $\alpha(G)=|G|+2$; if $iii)$ holds, $\alpha(G)$ is as shown in Table \ref{table:kd}.
\end{theorem}

The proof of Theorem \ref{203} is the subject of this section. Specifically, we compute the value of $\alpha(G)$ when $G$ is contained in one of the families of groups
mentioned in 
Theorem \ref{203}, $i),ii)$, in Propositions \ref{2t} and \ref{p2066}. Then, we calculate $\alpha(G)$ for the non-abelian groups $G$ that are shown in Table \ref{table:kd} in Propositions \ref{2c}, \ref{k3}, \ref{p5s} and \ref{arlu}.

Let us start by presenting two lemmas that will be used throughout the section.


\begin{lemma}\label{pq14}
Let $G$ be the dicyclic group of order $2^{r+1}q$, where $q$ is an odd prime or $q=1$. Let $\G$ be a graph such that $G\cong \Aut(\G)$ and consider the action of $G$ on the vertex set $V(\G)$. If every orbit has size at most $\te{\normalfont max}\{{2^{r+1},2^rq}\}$ then there exist at least 2 orbits of size $2^{r+1}$. 
\end{lemma}

\begin{proof} Let $G=\la y,x,b\mid y^{2^r}=x^q=1,y^{2^{r-1}}=b^2,yx=xy,byb^{-1}=y^{-1},bxb^{-1}=x^{-1} \ra.$\\
As the action of $G$ on $V(\G)$ is faithful, there exists a vertex $w$ of $\G$ such that $b^2\notin G_{w}$. Then, since $b^2$ is the only element of order 2, by Cauchy's Theorem, $2$ does not divide $|G_w|$. Therefore, by the orbit-stabilizer lemma, $2^{r+1}$ divides $|\mathcal{O}_w|.$ Thus $|\mathcal{O}_w|=2^{r+1}$, since $|\mathcal{O}_w|\leq \te{\normalfont max}\{{2^{r+1},2^rq}\}$. 
Suppose that $\mathcal{O}_w$ is the only orbit of size $2^{r+1}$.

Let $B=\{y^kbw\mid k\in \N\}$.
We will show that the map $\phi:V(\G)\rightarrow V(\G)$, where
\[ 
\phi(v)= {\left\{ 
\begin{array}{ll}
      b^2v,\;&\; \text{  if }v\in B,\\
      v, \;&\;\text{  if }v\notin
     B,
\end{array} 
\right. }\]
is an automorphism of $\G$. 
Indeed, the property $v_1\sim v_2\iff \phi(v_1)\sim \phi(v_2)$ holds for every $v_1,v_2\in V(\G)$ as
\begin{itemize}
    \item for $v_1,v_2\in B$, $ b^2\in \Aut(\G)$ hence $v_1\sim v_2\iff b^2v_1\sim b^2v_2$,
    \item for $v_1,v_2\notin B$, clearly $v_1\sim v_2\iff \phi(v_1)\sim \phi(v_2)$, 
    \item for $v_1=g_1w\notin B,v_2=g_2w\in B,g_1\in\la y,x\ra,g_2\in\la y,x\ra b$, we have $g_1w\sim g_2w\iff g_1g_2^{-1}g_1w \sim g_1w \iff b^2g_2w \sim g_1w$, since $ g_1g_2^{-1}\in \Aut(\G)$ and  $g_1g_2^{-1}g_1=b^2g_2$,
    \item for $v_1\notin \ma w,v_2\in B,$ we have that $v_1\sim v_2\iff v_1\sim b^2v_2$, as $b^2\in \Aut(\G) $ and $b^2\in G_{v_1}$, by the assumption that $2$ divides $|G_{v_1}|$.
\end{itemize}
We have reached a contradiction since
$G_w=G_{bw}=\la x\ra$. Hence there exists a second orbit of size $2^{r+1}.$
\end{proof}

\begin{lemma}\label{ika2}
Let $G=\Dic_q$, $q\in\{3,5\}$, and let $\G$ be a graph on at most $4q+4$ vertices such that $\Aut(\G)\cong G$. Then, there is no orbit of size $|G|=4q$ in the action of $G$ on $V(\G)$.
\end{lemma}

\begin{proof} Let $G=\la x,b \ra$, where $x^q=b^4=1$. 

Suppose that there is a vertex $v\in V(\G)$ with stabilizer $G_v=\{1\}$.
If there exists an orbit of size 4, let $u\in V(\G)$ be a vertex with stabilizer $G_u=\la x \ra$. By possibly replacing the graph $\G$ with its complement, $\overline{\G}$, we assume that $v$ is adjacent to up to two vertices in the orbit $\ma u$, if it exists. Without loss of generality, we also assume that if $v$ is adjacent to a vertex in $\ma u$ then
$v\sim u$. 
Let $B=\big\{x^kb^lv\mid k\in \N,l\in\{1,3\}\big\}$ and let $\phi:V(\G)\rightarrow V(\G)$,\vspace{-1 mm}
\[ 
\phi(w)={\left\{ 
\begin{array}{ll}
      b^2w,\; &\; \text{  if }w\in B,
      \\b^{(-1)^k l}w, \;&\;\text{  if }  w=b^ku \te{ and }
     v \sim b^lu,
       \te{ where}\;k\in\mathbb{N}, l\in\{1,3\},
          \\b^{-k}w, \;&\;\text{  if }  w=b^ku\te{ and }
      v \nsim b^lu,  \forall\;l\in\{1,3\},
      \te{ where}\;k\in\mathbb{N},\\
      w, \;&\;\text{  if }2 \text{  divides } |G_{w}| \text{ or }w\in \ma{_v}\setminus B.
\end{array} 
\right.} \]
We will show that $\phi\in \Aut(\G)$.
The property $v_1\sim v_2\iff \phi(v_1)\sim \phi(v_2)$ holds for every $v_1,v_2\in V(\G)$, as
\begin{itemize}
    \item if $v_1,v_2\in B$ then $v_1\sim v_2\iff b^2v_1\sim b^2v_2$, since $ b^2\in \Aut(\G)$,

      \item if $v_1,v_2$ are fixed by $\phi$ then clearly $v_1\sim v_2\iff \phi(v_1)\sim \phi(v_2)$, 
    \item if $v_1=g_1v\notin B,v_2=g_2v\in B,g_1\in \la x,b^2\ra,g_2\in \la x,b^2\ra b$, then $g_1v\sim g_2v\iff g_1g_2^{-1}g_1v \sim g_1v \iff b^2g_2v \sim g_1v$, since $ g_1g_2^{-1}\in \Aut(\G)$ and  $g_1g_2^{-1}g_1=b^2g_2$,
    \item  if $v_1\in \ma {u}$ then $v_2\in \ma {v}\cup \ma{u}$ 
($\ma{v}$ and $\ma{u}$ are the only orbits, as $|V(\G)|\leq 4q+4$);    $\phi$ was constructed to preserve adjacency and non-adjacency between $v_1$ and $ v_2$,
    \item if $v_1\in B$, $2$ divides $|G_{v_2}|$
    then $b^2\in G_{v_2} $, hence $v_1\sim v_2\iff b^2v_1\sim v_2$.
\end{itemize}
We have reached a contradiction, since $\phi$ fixes $v$ but not $bv$ and $G_v=G_{bv}=\{1\}$.
\end{proof}
Using Lemma \ref{pq14} we recover the following result, which was first proven in \cite{grav}.
\begin{proposition}{\label{2t}}
The generalised quaternion group $Q_{2^{r+1}}$ satisfies $\alpha (Q_{2^{r+1}})=2^{r+2}.$
\end{proposition}
\begin{proof}
By Babai's Theorem \ref{30}, $\alpha (Q_{2^{r+1}})\leq 2^{r+2}$.
The inequality $\alpha (Q_{2^{r+1}})\geq 2^{r+2}$ follows from Lemma \ref{pq14}.
\end{proof}

\begin{proposition}{\label{p2066}}
The generalised dicyclic group $Q_{2^{r+1}}\times C_2$ satisfies $\alpha(Q_{2^{r+1}}\times C_2)= 2^{r+2}+2$.
\end{proposition}

\begin{proof}Let $G=\la y,x,b\mid y^{2^r}=x^2=1,y^{2^{r-1}}=b^2,yx=xy,bx=xb,byb^{-1}=y^{-1}\ra.$

Let $\G$ be a graph on at most $2^{r+2}+1$ vertices with automorphism group isomorphic to $G$.
The faithfulness of the action of $G$ on $V(\G)$ implies the existence of some $w\in V(\G)$ such that $b^2 \notin G_w$. Then $G_w\in\{\la 1\ra,\la x\ra,\la b^2x\ra\}$. 
Let $u\in V(\G)$ be such that $G_u\in \{\la x\ra,\la b^2x\ra\}$ and $u\notin \ma w$; if no such vertex exists, let $u=w$. Since $|V(\G)|\leq 2^{r+2}+1$, there exist at most two orbits of size $2^{r+1}$ or one of size $2^{r+2}$. 
Therefore, if $u\neq w$ then $G_u\neq G_w$, as the action of $G$ on $V(\G)$ is faithful.

Let $B=\big\{y^kx^lbz\mid z\in {\{w,u\}}, k,l\in \mathbb{N}\big\}$ and let $\phi:V(\G)\rightarrow V(\G)$ be the map
\[ 
\phi(v)= {\left\{ 
\begin{array}{ll}
    b^2v,\;&\; \text{  if }v\in B,\\
      v, \;&\;\text{  if }v\notin
     B.
\end{array} 
\right.} \]
We will show that $\phi$ is an automorphism of $\G$ by proving that
$v_1\sim v_2\iff \phi(v_1)\sim \phi(v_2)$, for every $v_1,v_2\in V(\G)$. Indeed,
\begin{itemize}
    \item if $v_1,v_2\in B$, then $v_1\sim v_2\iff b^2v_1\sim b^2v_2$,
    since $ b^2\in \Aut(\G)$,
    \item if $v_1,v_2\notin B$, then clearly$\;v_1\sim v_2\iff \phi(v_1)\sim \phi(v_2)$,
 \item if $v_1=g_1z\notin  B, v_2=g_2z\in B, z\in \{w,u\}, g_1\in \la y,x\ra,g_2\in \la y,x\ra b$, then $g_1z\sim g_2z\iff g_1g_2^{-1}g_1z \sim g_1z \iff b^2g_2z \sim g_1z$, as $ g_1g_2^{-1}\in \Aut(\G), \;g_1g_2^{-1}g_1=b^2g_2$,
\item if $v_1\notin B,v_2\in B$ and $G_{v_1}=\la b^{2n}x\ra,G_{v_2}=\la b^{2(n+1)}x \ra$ for some $n\in\{1,2\}$, then we have
 $v_1\sim v_2 \iff v_1\sim (b^{2n}x) (b^{2(n+1)}x)v_2\iff v_1\sim b^2v_2$,
    \item if $v_1\notin \ma w\cup \ma u$, $v_2\in B$, then $v_1\sim v_2 \iff v_1\sim b^2v_2$, since $b^2\in G_{v_1}$.
\end{itemize}
The map $\phi$ fixes $w $ but not $ bw$ and $G_{w}=G_{bw}$; a contradiction. Thus, $\alpha(G)\geq 2^{r+2}+2$.

Let us now construct a graph $\G$ on $2^{r+2}+2$ vertices such that $\Aut(\G)\cong G$.
\begin{construction}
Let $\G_1$ be a graph on $2^{r+2}$ vertices constructed according to Babai's Construction \ref{con0} for the generalised quaternion group $ Q_{2^{r+1}}$ and $\G_2$ be the connected graph on 2 vertices. 
The graph $\G=\G_1\cup \G_2$ has $\Aut(\G)\cong G$, since it consists of two connected components,
$\G_1,\G_2$, of different size and
$\Aut(\G_1)\cong  Q_{2^{r+1}},\Aut(\G_2)\cong C_2$.
\end{construction}
\vspace{-9 mm}
\end{proof}

\begin{proposition}{\label{2c}}
The generalised dicyclic group $G_{16}=\la x,b\mid x^{4}=b^{4}=1,\;bxb^{-1}=x^{3}\ra$ satisfies $\alpha(G_{16})=18$. 
\end{proposition}\vspace{-0.3 cm}
We will prove Proposition \ref{2c} using the following lemma.\vspace{-0.3 cm}
\begin{lemma}{\label{p0}}
Suppose that $\G$ is a graph on at most 17 vertices such that $\Aut(\G)\cong G$, where $G=G_{16}$, and consider the action of $G$ on $V(\G)$.
Then, there is no vertex with stabilizer equal to
$\la x^2b^2\ra$. Moreover, there are two orbits, $\ma {v_1
}, \ma {v_2}$,  such that $G_{v_1},G_{v_2}\in\big\{\la x\ra, \la b^2x\ra\big\}$.
\end{lemma}\vspace{-0.4 cm}

\begin{proof}
The action of $G$ on $V(\G)$ is faithful; thus, there exists some vertex $v_1\in V(\G)$ such that $b^2\notin G_{v_1}$. Let $v_1,v_2,...,v_s\in V(\G)$ form a maximal set of vertices such that
$b^2\notin G_{v_i},$ for $i\in\{1,...,s\}$, and
 the orbits $\ma {v_1},...,\ma {v_s}$ are distinct. Since $ |\ma {v_i} |\geq 4$, for  $i\in\{1,...,s\}$, the assumption $ |V(\G)|\leq 17$ implies that
$s\leq 4$.

Suppose
that 
there do not exist distinct $i,j\in\{1,...,s\}$ such that $G_{v_i},G_{v_j}\in \big\{\la x\ra, \la b^2x\ra\big\}$ or there exists $i\in\{1,...,s\}$ such that $G_ {v_i}=\la x^2b^2\ra$. 
Since $|V(\G)|\leq 17$ and the action of $G$ on $V(\G)$ is faithful, if $s \geq 3$ then there exists $i\in\{1,...,s\}$ such that $x^2\notin G_{v_i}$; hence $G_{v_i}=\la x^2b^2\ra$. Moreover, since the action is faithful, there is at most one $i\in\{1,...,s\}$ such that  $G_{v_i}=\la x^2b^2\ra$. To sum up, we have $s\leq 2$ or $G_{v_i}=\la x^2b^2\ra$ for a unique $i\in\{1,...,s\}$.
If the latter is true, without loss of generality let $v_1$ have stabilizer $G_{v_1}=\la x^2b^2 \ra$; alternatively, let $v_1$ be such that
 $|\ma {v_1}|\geq |\ma {v_i}|$, for $ i\in\{1,s\}$. Moreover, 
if $s=2$, without loss of generality we assume that if $v_1$ is connected to $\ma {v_2}$ then $v_1\sim v_2$. Finally, by possibly replacing $\G$ with its complement, let $v_1$ be adjacent to at most half the vertices of $\ma {v_2}$.

Let
$B=\big\{x^kb^lv_1\mid k\in \mathbb{N},l\in\{1,3\}\big\}$.
We will show that the map $\phi:V(\G)\rightarrow V(\G)$, 
\vspace{-1 mm}\[ 
\phi(v)= {\left\{ 
\begin{array}{ll}
    b^2v, \;&\;\text{  if }  v\in B,
  \\v, \;&\;\text{  if }   v\in\ma {v_2} \text{  and }  G_{v_1}=\la x^2b^2\ra,
      \\b^{(-1)^k l}v, \;&\;\text{  if }  v=b^kv_2,
       v_1 \sim b^lv_2,
       \te{ where}\;k\in\mathbb{N}, l\in\{1,3\}, \te{ and }G_{v_1}\neq\la x^2b^2\ra,
          \\b^{-k}v, \;&\;\text{  if }  v=b^kv_2,
      v_1 \nsim b^lv_2,  \forall\;l\in\{1,3\},
      \te{ where}\;k\in\mathbb{N},
      \te{ and }G_{v_1}\neq\la x^2b^2\ra,
      \\v, \;&\;\text{  if }  v \notin (B\cup \ma {v_2}),
      
\end{array} 
\right. }\]
is an automorphism of $\G$. 
Indeed, $u_1\sim u_2\iff \phi(u_1)\sim \phi(u_2)$, for all $u_1,u_2\in V(\G)$, as
\begin{itemize}
    \item if $u_1,u_2\in B$ then $u_1\sim u_2\iff b^2u_1\sim b^2u_2$, since $b^2\in \Aut(\G)$,
    \item if $u_1,u_2$ are fixed by $\phi$ then clearly $u_1\sim u_2\iff \phi(u_1)\sim \phi(u_2)$, 
    \item if $u_1=g_1v_1\notin B, u_2=g_2v_1\in B, g_1\in \la x,b^2\ra,g_2\in \la x, b^2\ra b$, then $g_1v_1\sim g_2v_1\iff g_1g_2^{-1}g_1v_1 \sim g_1v_1 \iff b^2g_2v_1 \sim g_1v_1$, since $g_1g_2^{-1}\in \Aut(\G)$ and  $g_1g_2^{-1}g_1=b^2g_2$,
    \item  if $u_1\in \ma {v_2},u_2\in \ma {v_1}\cup \ma{v_2}$ and $G_{v_1}\neq \la x^2b^2\ra$, 
    then $\phi$ was constructed to preserve adjacency and non-adjacency between $u_1$, $u_2$,
    \item  if $u_1\in B,u_2\in \ma {v_2}$ and $G_{v_1}=\la x^2b^2\ra$, then 
    $u_1\sim u_2\iff x^2(x^2b^2)u_1\sim u_2\iff b^2u_1\sim u_2$
    since $x^2b^2\in G_{v_1},x^2\in G_{v_2}$,

    \item if $\phi(u_1)=b^2u_1$ and $ {u_2}\notin \ma {v_i}$ for all $i\in\{1,...,s\}$, then  $u_1\sim u_2\iff b^2u_1\sim u_2$, as $b^2\in G_{u_2}$,
  \item
  if $\phi(u_1)=b^lu_1$ for some $l\in\{1,3\}$, and $u_2\notin \ma {v_i}$ for all $i\in\{1,...,s\}$, then, by assumption, $G_{v_1}=\la x^2\ra$ and $G_{v_2}=G_{u_1}\in \{\la x\ra, \la b^2x\ra\}$. By faithfulness, there is $w\in V(\G)$, $x^2\notin G_{w}.$ Since $|V(\G)|\leq 17$, $G_{w}= \la x^kb\ra,$ for some $k\in \mathbb{N}$. In any case,  $G_{u_1}\in \{\la x\ra, \la b^2x\ra\}$, $G_{u_2}\in \{\la x^kb\ra \mid k\in \mathbb{N}\}\cup G$ imply that $u_1\sim u_2 \iff b^lu_1\sim u_2$.
\end{itemize}
We have reached a contradiction since
$G_{v_1}=G_{bv_1}$ and $\phi$ fixes $v_1$ but not $bv_1$.
\end{proof}\vspace{-3 mm}
\begin{proof}[Proof of Proposition \ref{2c}]

Let $G={G_{16}}$. Suppose that $\G$ is a graph on at most 17 vertices such that $\Aut(\G)\cong G$.
As the action of $G$ on $V(\G)$ is faithful, there exists $w_1\in V(\G)$ such that $x^2\notin G_{w_1}$; by Lemma \ref{p0}, $G_{w_1}\neq \la1\ra,\la x^2b^2\ra$, hence $G_{w_1}= \la b^2\ra$ or $G_{w_1}= \la x^kb\ra$ for some $k\in \mathbb{N}$. Let $\ma {v_1},\ma {v_2}$ be two distinct orbits such that $G_{v_1},G_{v_2}\in \big\{\la x\ra, \la b^2x\ra\big\}$, which exist by Lemma \ref{p0}. Since $|V(\G)|\leq 17$, there are up to two orbits, of total size at most eight, containing vertices that are not fixed by $x^2$.
If there are exactly eight such vertices, let
$w_2\in V(\G)$ be such that $x^2\notin G_{w_2}$ and
the vertices $\{w_1,xw_1,x^2w_1,x^3w_1,w_2,xw_2,x^2w_2,x^3w_2\}$ are distinct; if only four vertices of $\G$ are not fixed by $x^2$, let $w_2=w_1$. By possibly replacing $\G$ with its complement, we assume that the vertex $w_1$ is adjacent to up to two vertices of the set $\{w_2,xw_2,x^2w_2,x^3w_2\}$. Without loss of generality, if $w_1\neq w_2$, let these vertices be $w_2$ and $x^{\delta}w_2$, $\delta\in\{0,1,2,3\}$;
if $w_1$ is adjacent to exactly one vertex of the set $\{x^kw_2\mid k\in \mathbb{N}\}$ or $w_1=w_2$, let $\delta=0$.
Then \vspace{-2 mm}\begin{equation}\label{9e}w_1\sim x^nw_{2}\iff w_1\sim x^{\delta-n}w_{2},\;\;\forall\;n\in\mathbb{N}.\end{equation}\vspace{-2 mm}
We will show that the map $\psi:V(\G)\rightarrow V(\G)$, where
\vspace{-1 mm}\[ 
\psi(v)= {\left\{ 
\begin{array}{ll}
    x^{-k}w_1,\;&\;\text{     if }v=x^kw_1, \;k\in \mathbb{N},\\          x^{\delta-l}w_{2},\;&\;\text{     if }v=x^lw_{2}, \;\;l\in \mathbb{N},\\
          v,\;&\;\text{     if }x^2\in G_v,
\end{array} 
\right.} \]\vspace{-1 mm}
is an automorphism of $\G$. Indeed, $u_1\sim u_2\iff \psi(u_1)\sim \psi(u_2)$ for all $u_1,u_2\in V(\G)$, as
\begin{itemize}
    \item if $u_1=x^kw_i, u_2=x^lw_i, k,l\in \N, {i\in \{1,2\}}$, then
    $x^{k}w_i\sim x^{l}w_i\iff x^{j-l}w_i\sim x^{j-k}w_i,$ for $j=0,\delta,$ since $x^{j-k-l}\in \Aut(\G)$,
    \item if $u_1,u_2$ are fixed by $\psi$ then clearly $u_1\sim u_2 \iff\psi(u_1)\sim \psi(u_2)$,
    \item if $u_1=x^kw_1, u_2=x^lw_2$ for $k,l\in \N$ and $w_1\neq w_2$ then,
    by (\ref{9e}),
    $x^kw_1\sim x^{l}w_2\iff w_1\sim x^{l-k}w_2\iff  w_1\sim x^{\delta-l+k}w_2\iff x^{-k}w_1\sim x^{\delta-l}w_2$,
    \item if $u_1\in \ma {v_1}\cup\ma {v_2}$ or $|\ma {u_1}|=1$, and ${u_2}\in \ma {w_1}\cup \ma {w_2},$ then $u_1\sim u_2 \iff u_1\sim x^ku_2,$ for all  $k\in \mathbb{N}$, since $G_{u_1}\in\{\la x\ra,\la b^2x\ra,G\}$ and $ b^2\in G_{u_2}$,
    \item if $|\ma {u_1}|=2$ and ${u_2}=x^k w_1, k\in \{1,3\}$, then $x^2\in G_{u_1}$ hence $u_1\sim x^kw_1 \iff u_1\sim x^{k+2}w_1\iff \psi(u_1)\sim \psi(u_2)$; note that $w_1=w_2$, since $|V(\G)|\leq 17$.
\end{itemize}
We have reached a contradiction:
$\psi$ fixes $v_1$ and $w_1$ but $G_{v_1}\cap G_{w_1}=\{1\}$. Thus $\alpha(G)\geq 18$.

We will complete the proof by constructing a graph $\G$ on 18 vertices having $\Aut(\G)\cong G$. \vspace{-8 mm}
\begin{construction}{\label{con200}}\vspace{-1 mm}
Let $\G$ be a graph with vertex set $V(\G)=\ma 1\cup 
\ma {1'}\cup \ma {1''} \cup \ma {1'''}$, where
$\ma 1=\{1,2,...,8\},\;\ma {1'}=\{1',2'\},\;\ma {1''}=\{1'',2'',3'',4''\}$ and $\ma {1'''}=\{1''',2''',3''',4'''\}$.
We define the edge set of $\G$ to be such that,
for $v,w$ $\in \ma 1$, 
\vspace{ -2 mm}\begin{align*}
v'\nsim w';  &\;v'\nsim w''; \; v'\nsim w''';\;v'''\nsim w''';\\
v\sim w' &\iff v-w\equiv 0\;\te{(mod }2);
\\
v''\sim w''' &\iff  w-v\equiv 0,1\;\te{(mod }4);\\
v''\sim w'' &\iff v-w\equiv 2\;\te{(mod }4);
\\
v\sim w'' &\iff  (v,w)\in\bigcup_{k\in\{0,1\}}\big(\{4k+i\mid 1\leq i\leq 4\}\times\{k+1,k+3\}\big);
\\
v\sim w'''&\iff v\sim w'';
\\
v\sim w&\iff w-v\equiv0,(-1)^k\;\te{(mod }4), \qquad v\neq w,\qquad \te{and} \\&k\in\{0,1\} \te{ is such that }v\in \{4k+i\mid 1\leq i\leq 4\},w\notin \{4k+i\mid 1\leq i\leq 4\}.
\end{align*}
\end{construction}\vspace{-2 mm}
Using the mathematical software
GAP \cite{gap} we verified that $\Aut(\G)\cong G$.
\end{proof}\vspace{-2 mm}

\begin{proposition}{\label{k3}}
For the groups $\Dic_3,\; \Dic_5,\; \Dic_6,\;  Q_8\times C_3$ we have that $\alpha(\Dic_3)=17$, $\alpha(\Dic_5)=23$, and $\alpha(\Dic_6)=\alpha(Q_8\times C_3)=25$.
\end{proposition}
\begin{proof}
Let $G=\la y,x,b\ra$ be a generating set for $G$ such that $y^{2^r}=x^q=1, y^{2^{r-1}}=b^2,yx=xy$, where $r\in\{1,2\},q\in\{3,5\}$.

Suppose that there exists a graph $\G$ such that
$\Aut(\G)\cong G$
and $|V(\G)|< 3q+2^{r+2}$. We will study the action of the group $G$ on the vertex set of $\G$, in order to reach a contradiction.
 Let us start by showing that at least $3q$ vertices of $\G$ are not fixed by $x$.

Suppose, in contrast, that less than $3q$ vertices are not fixed by $x$.
The action of $G$ on $V(\G)$ is faithful, hence there exists some vertex $v\in V(\G)$ such that $x\notin G_v$.
Then $q$ does not divide $|G_v|$, thus $q$ divides $|\mathcal{O}_v|$. Hence there are at most two orbits of size $q$, or one of size $2q$. 
Let $V=\big\{x^kz\mid  z\in \{v,u\}, k\in \mathbb{N}\big\}$ be the set of vertices of $\G$ that are not fixed by $x$, where $u\notin\{ x^kv\mid k\in \mathbb{N}\}$, if there exist
two orbits of size $q$ or one of size $2q$,
and $u=v$, if there is exactly one orbit of size $q$.

By possibly replacing $\G$ with its complement, we assume that $v$ is adjacent to up to two vertices of the set $\{x^ku\mid k\in \mathbb{N}\}.$ Without loss of generality, if $v\neq u$, we consider these vertices to be $u$ and $x^\delta u,$ where $\delta\in \{0,1,...,q-1\}$; if $v$ is adjacent to exactly one vertex of the set $\{x^ku\mid k\in \mathbb{N}\}$ or $v=u$, let $\delta=0$. In any case, we have \begin{equation}\label{de3}v\sim x^nu\iff v\sim x^{\delta-n}u,\;\;\forall\;n\in\mathbb{N}.\end{equation}
We will show that the map $\psi:V(\G)\rightarrow V(\G)$, where \[\psi(z)=  {\left\{
\begin{array}{ll}
          x^{-k}v,\;&\;\text{     if }z=x^kv, \;k\in \mathbb{N},\\
          x^{\delta-l}u,\;&\;\text{     if }z=x^lu, \;l\in  \mathbb{N},\\
          z,\;&\;\text{     if }q\text{     divides }| G_z|,\\
\end{array} 
\right. }
\]
is an automorphism of $\G$ by proving that $v_1\sim v_2\iff \psi(v_1)\sim \psi(v_2)$, for every $v_1,v_2\in V(\G)$. Indeed,
\begin{itemize}
    \item if $v_1=x^kz,v_2=x^lz, z\in\{v,u\},k,l\in \N,$ 
    then $x^kz\sim x^lz\iff x^{j-k-l}x^{k}z\sim x^{j-k-l}x^lz\iff x^{j-l}z\sim x^{j-k}z$, for $j=0,\delta$,
    \item if $v_1,v_2$ are fixed by $\psi$ then clearly $v_1\sim v_2 \iff\psi(v_1)\sim \psi(v_2)$,
    \item if $v_1=x^kv,v_2=x^lu,k,l\in \N$, then,
    by (\ref{de3}),
    $x^kv\sim x^{l}u\iff v\sim x^{l-k}u\iff v\sim x^{\delta-l+k}u\iff x^{-k}v\sim x^{\delta-l}u$,
     \item if $v_1\notin \ma v\cup \ma u$, $v_2\in \ma v\cup \ma u$ then $v_1\sim v_2 \iff v_1\sim x^kv_2,\forall k\in \mathbb{N}$, as
     $\la x\ra \leq G_{v_1}$.
\end{itemize}
As the action of $G$ on $V(\G)$ is faithful, there exists some
$w\in V(\G)$ such that $b^2\notin G_w$, hence
$G_w=\la x \ra$ ($G_w\neq \{1\}$, since there exist less than $3q$ vertices that are not fixed by $x$). Then, $\psi\in G_w\cap G_{v}=\{1\}$ and $\psi(xv)=x^{-1}v$; a contradiction.

Thus, at least $3q$ vertices of $\G$ are not fixed by $x$. We will show that there is no orbit of size $|G|$. Indeed, if $|G|=24$ then we assumed that $|V(\G)|\leq 24$ hence, by the GRR theorem, there exists no orbit of size 24. If $|G|=12,20$ then, by
Lemma \ref{ika2}, no orbit has size equal to $|G|$. Therefore, every orbit has size at most $2^rq$.
By Lemma \ref{pq14}, there exist at least two orbits of size $2^{r+1}$ (the statement of  Lemma \ref{pq14} also holds for the group $G=Q_8\times C_3$ and the proof is analogous). Considering the number of vertices that are fixed or not fixed by $x$ we conclude that $|V(\G)|\geq 3q+2^{r+2}$; a contradiction. Hence, $\alpha(G)\geq 3q+2^{r+2}$.

Let us now consider each case for $G\in\big\{\Dic_3,\Dic_5,\Dic_6,Q_8\times C_3\big\}$ and
 construct a graph $\G$ such that $\Aut(\G)\cong G$,
completing the proof that 
$\alpha(G)= 3q+2^{r+2}$.
\begin{construction}
Assume that $G=\Dic_q$ for some $q\in\{3,5\}.$
Let $\G$ be a graph with vertex set $V(\G)=\ma 1  \cup 
\ma {1'} \cup \ma {1''}\cup \ma {1'''}$, where
$\ma 1=\{1,2,...,2q\},\;\ma {1'}=\{1',2',...,q'\},\;\ma {1''}=\{1'',2'',3'',4''\},\;\ma {1'''}=\{1''',2''',3''',4'''\}$,
and edge set such that,
given $v,w\in \ma 1$,
\vspace{ -2 mm}\begin{align*}
v'\nsim w'; &\; v'''\nsim w''';\; v'\nsim w'';\; v'\nsim w''';  \\
v\sim w &\iff w-v\equiv 0\;\te{(mod }q),\qquad v\neq w;
\\
v''\sim w''&\iff w-v\equiv 1\;\te{(mod }4) ;
\\
v\sim w' &\iff w-v\equiv 0\;\te{(mod }q), v\leq q,\qquad \te{or }w-v\equiv 1\;\te{(mod }q), v>q ;
    \\
v\sim w''&\iff w-v\equiv 0\;\te{(mod }2), v>q, \qquad\te{or }w-v\equiv 1\;\te{(mod }2) ,v\leq q ;
\\v\sim w'''&\iff v\sim w'';
\\
v''\sim w'''&\iff w-v\equiv 0,1\;\te{(mod }4) .
\end{align*}
\end{construction}\vspace{ -2 mm}
\begin{construction}{\label{con300}}
For $G=\Dic_6$, we let $\G$ be the graph with vertex set $V(\G)=\ma 1  \cup 
\ma {1'} \cup \ma {1''}\cup \ma {1'''}$, where
$\ma 1=\{1,2,...,8\},\;\ma {1'}=\{1',2',...,8'\},\;\ma {1''}=\{1'',2'',...,6''\},\;\ma {1'''}=\{1''',2''',3'''\}$,
and edge set such that,
for $v,w\in \ma 1$,
\vspace{ -2 mm}\begin{align*}v\nsim w; &\; v'''\nsim w'''; 
\;v\nsim w''';\; v'\nsim w'';\;v'\nsim w'''; \\
v'\sim w' &\iff w-v\equiv 4\;\te{(mod }8);
\\
v''\sim w'' &\iff w-v\equiv 1\;\te{(mod }3), v\leq 3,w>3,\qquad \te{or }w-v\equiv 2\;\te{(mod }3), v>3, w\leq 3 ;
\\
v\sim w'&\iff w-v\equiv 0\;\te{(mod }8), \qquad\te{or }w-v\equiv 3\;\te{(mod }4) \te{ and } (v\leq 4 \iff w\leq 4) ;
\\
v\sim w''&\iff v\leq 4,w\leq 3, \qquad\te{or }v> 4,w> 3;
\\
v''\sim w'''&\iff w-v\equiv 0\;\te{(mod }3) .
\end{align*}
\end{construction}
\vspace{ -2 mm}\begin{construction}{\label{con300}}
Finally, for $G=Q_8\times C_3$, let $\G_1$ be a graph on $16$ vertices constructed according to Babai's Construction \ref{con0} for the group $ Q_{8}$ and let $\G_2$ be 
a graph on 9 vertices such that $\Aut(\G_2)\cong C_3$, which exists by Proposition \ref{27b}. We let 
$\G=\G_1\cup \G_2$.
\end{construction}\vspace{ -2 mm}
Using GAP \cite{gap}, we confirmed that each graph has the desired automorphism group.
\end{proof}
\vspace{ -2 mm}

\begin{proposition}\label{p5s}
Let  $G=\la a,b\mid a^{8}=b^{2}=1,\;bab^{-1}=a^{5}\ra$. Then $\alpha(G)=18$.
\end{proposition}

Suppose that there exists a graph $\G$ on at most 17 vertices having $\Aut(\G)\cong G$ and consider the action of $G$ on the vertex set $V(\G)$.
As the action is faithful, there exists a vertex $w\in V(\G)$ such that $a^4\notin G_w$, hence $G_w\in\large\{\la 1\ra,\la a^4b\ra,\la b\ra\large\}$. If $G_w=\la 1\ra$ then the subgraph $\G_1$ of $\G$ that is induced by $\ma w$, has order 16 and automorphism group $\Aut(\G_1)\cong G$, which contradicts the fact that $G$ has no GRR. Therefore, since the groups $\la a^4b\ra,\la b\ra$ are conjugate, we assume that $G_w=\la b\ra$. Let us define the vertex
$u\in V(\G)$ to be such that $u\notin \ma w$ and $G_u=\la b\ra$, if there exists a second orbit that its elements are not fixed by $a^4$; if no such orbit exists, we let $u=w$. 


\begin{lemma}{\label{triole}}
The vertex $w$ is adjacent to exactly one of the vertices $a^2 u,$ $a^6u$ in $\Gamma$, with $\G$ as above.
\end{lemma}
\begin{proof}
Suppose, conversely, that \vspace{-0.5 mm} \begin{equation}\label{10e}
 w\sim a^2u \iff w\sim a^6u.\end{equation}
Let $B=\{w,a^4w,u,a^4u\}$. Define $\phi:V(\G)\rightarrow V(\G)$ to be the map 
\[ 
\phi(v)={ \left\{
\begin{array}{ll}
     a^4v, \;&\;\text{ if }v\in B,\\
     v, \;&\;\text{ if }v\notin  B.
\end{array} 
\right. }
\]
We will show that $v_1\sim v_2\iff \phi(v_1)\sim \phi(v_2),$ for every $v_1,v_2\in V(\G)$. Indeed,
\begin{itemize}
 \item if $v_1,v_2\in B$ then $v_1\sim v_2\iff a^4v_1\sim a^4v_2,$ since $a^4\in \Aut(\G)$,
    \item if $v_1,v_2\notin B$ then clearly $v_1\sim v_2\iff \phi(v_1)\sim \phi(v_2),$
 \item if $v_1=a^{4k+2}v,v_2=a^{4l}v,k,l\in \N,v\in \{u,w\}$, then we have $a^{4k+2}v\sim a^{4l}v\iff a^{4k+2-4l}a^{4k+2}v\sim a^{4k+2-4l}a^{4l}v\iff a^{4(l+1)}v\sim a^{4k+2}v\iff a^4v_2\sim v_1$,
   \item if $v_1=a^{4k+2}v,v_2\in B\setminus\{v,a^4v\},k\in \N,v\in \{u,w\}$, then, by  (\ref{10e}),  $v_1\sim v_2\iff v_1\sim a^4v_2,$
   \item if $v_1=a^{2k+1}v,v_2\in B,k\in \N, v\in \{u,w\}$, then$\;v_1\sim v_2\iff v_1\sim a^4v_2$, since $a^4b\in G_{v_1}$ and $b\in G_{v_2}$,
    \item if $v_1\notin \ma {u}\cup \ma {w}$, $v_2\in B$ then $v_1\sim v_2\iff v_1\sim a^4v_2$, since $a^4\in G_{v_1}$.
    \end{itemize}
We conclude that $\phi$ is an  automorphism of $\G$ that fixes $a^2w$ but not $w$, which contradicts the equality
$G_{w}=G_{a^2w}$.
\end{proof}

\begin{proof}[Proof of Proposition \ref{p5s}]
If $\G$ is a graph on at most 17 vertices having $\Aut(\G)\cong G$  and  $\ma w$, $\ma u$ are the orbits of size 8 described above then,
by Lemma \ref{triole}, 
either $w\sim a^2u$ or $w\sim a^6u$ (hence $w\neq u$).
Arguing in a similar way, 
we can show that $w\sim u \iff w\nsim a^4u$. Without loss of generality, let us assume that $w\sim u,w\sim a^2u.$
Moreover, since $b\in G_w, b\in G_u,$ we have that $ w\sim a^3u \iff w\sim ba^3bu \iff w\sim a^7u$.
Therefore, the statement
\begin{equation}\label{11e}
w\sim a^nu\iff w\sim a^{2-n}u\end{equation}
 holds for every $n\in \mathbb{N}.$
 
Since $|V(\G)|\leq 17$ and $|\ma{w}\cup\ma{u}|=16$, there is at most one additional orbit, which has size 1.
Let $ \psi:V(\G)\rightarrow V(\G) $ be the map 
\[ 
\psi(v)=  {\left\{
\begin{array}{ll}
          a^{-k}u, \;&\;\text{ if }v=a^ku, \;k\in \N,\\
          a^{6-l}w, \;&\;\text{ if }v=a^lw, \;l\in \N,\\
          v, \;&\;\text{ if }|\ma v|=1
          .
\end{array} 
\right. }
\]
We will show that $\psi\in \Aut(\G)$. Indeed, 
$v_1\sim v_2\iff \psi(v_1)\sim \psi(v_2)$, for every $v_1,v_2\in V(\G)$, since
\begin{itemize}
 \item for $v_1=a^kv,v_2=a^lv,v\in \{u,w\}, k,l\in \mathbb{N}$, we have that $a^kv\sim a^lv\iff a^{j-k-l}a^{k}v\sim a^{j-k-l}a^lv\iff a^{j-l}v\sim a^{j-k}v$, for $j=0,6$,
        \item for $v_1=a^ku,v_2=a^lw,k,l\in \N$, we have that, by (\ref{11e}), $a^ku\sim a^lw\iff a^{k-l}u\sim w\iff a^{2-(k-l)}u\sim w\iff a^{-k}u\sim a^{6-l}w$,
          \item for $v_1,v_2$ such that $|\ma {v_1}|=1$, we have $v_1\sim v_2\iff v_1\sim a^kv_2,$ for every $k\in \mathbb{N}$, as $\la a\ra \leq  G_{v_1}$.
\end{itemize}
We have reached a contradiction since
$G_u=G_{a^2u}$ and $\psi$ fixes $u$ but not $a^2u$.

We will complete the proof by constructing a graph $\G$ on 18 vertices having automorphism group $\Aut(\G)\cong G$. 
\begin{construction}{\label{con300}}
Let $\G$ be a graph with $V(\G)=V_1\cup 
V_2\cup V_{3}$, where
$V_1=\{1,2,...,8\},\;V_{2}=\{1',2'...,8'\}$ and $V_{3}=\{1'',2''\}$.
We define the edge set of $\G$ to be such that,
for $v,w\in V_1$, 
\vspace{-8 mm}\begin{align*}
&v''\nsim w'';\\
v\sim w &\iff w-v\equiv 1,7\;\te{(mod }8);
\\v'\sim w' &\iff w-v\equiv 3,5\;\te{(mod }8);
\\v\sim w'&\iff w-v\equiv 0,1,3\;\te{(mod }8);
\\v\sim w''&\iff w=1;
\\v'\sim w''&\iff w=2.
\end{align*}
\end{construction}\vspace{-3 mm}
Using GAP \cite{gap} we computed that $\Aut(\G)\cong G$.
\end{proof}

\begin{proposition}{\label{arlu}}
The alternating group $A_4$ satisfies $\alpha(A_4)=16$.
\end{proposition}
Let $G=\langle a,b\rangle $, where $a=(1\;2\;3)$ and $b=(1\;2)(3\;4).$
Suppose that there exists a graph $\G$ on at most 15 vertices with $\Aut(\G)\cong G$.  

Since $G$ has no subgroup of order 6, there is no orbit of size 2.
Furthermore, there exists $z\in V(\G)$ such that $b\notin G_z$.
Then $|G_z|\in \{1,2,3\}$, hence the orbit $\ma z$ has size 4, 6 or 12. 
We examine each of these cases.
\begin{lemma}\label{opqt}
Let $G=A_4$ and let $\G$ be as in the previous proposition and consider the action of $G$ on $V(\G)$. Then, there exists no orbit of size 4.
\end{lemma}
\begin{proof}
Assume, in contrast, that there exists some orbit of size 4.
If there also exists an orbit of size 6 as well as an orbit of size 3, then without loss of generality we let $w\in V(\G)$ be such that $|\ma w|=3$ and $u\sim aw\iff u\sim a^2w$, for every $u\in V(\G)$ having $G_u=\la b\ra$; this is possible since the bound $|V(\G)|\leq 15$ ensures that there is at most one orbit of size 6. Otherwise, let $w$ be such that $G_w=\la aba\ra$.
Let us construct the map  $\phi:V(\G)\rightarrow V(\G)$, where 
\[ 
\phi(v)= { \left\{
\begin{array}{ll}
     a^2v,\;&\;\text{     if }\;G_v=\langle b\rangle ,G_v=\langle aba\rangle \text{ or   $v=w$}, \\
     av,\;&\;\text{     if }\;G_v=\langle a^2ba\rangle ,G_v=\langle ab\rangle\text{ or $v=a^2w$}, \\
     v,\;&\;\text{     otherwise}.
\end{array} 
\right. }\]
Practically, $\phi$ interchanges two pairs of vertices in the orbit of size 6, if it exists, one pair in the orbit of size 3, if both an orbit of size 3 and 6 exist, and one pair in every orbit of size 4. We will show that $\phi$ is an automorphism of $\G$.
The property  $v_1\sim v_2 \iff \phi(v_1)\sim \phi(v_2)$ holds for every $v_1,v_2\in V(\G)$, since 
\begin{itemize}
     \item if $\phi(v_1)=v_1,\phi(v_2)=v_2$ then clearly $v_1\sim v_2 \iff \phi(v_1)\sim \phi(v_2)$,
    \item if $(|\ma {v_1}|,|\ma {v_2}|)=1$ then there are either no edges or all possible edges between $\ma {v_1}$ and $\ma {v_2}$, hence $v_1\sim v_2 \iff \phi(v_1)\sim \phi(v_2)$, 
    \item if $|\ma {v_1}|=3,|\ma {v_2}|=6$ then there is exactly one orbit of size 3 and one of size 6; $\phi$ was constructed to preserve adjacency and non-adjacency between $v_1$ and $ v_2$,
    \item if $v_1,v_2\in\ma w$ and $|\ma w|=3$, then the subgraph induced by $\ma w$ is either a clique or empty, thus $v_1\sim v_2 \iff \phi(v_1)\sim \phi (v_2)$,
    \item if
    $G_{v_1}=\la aba^2 \ra,G_{v_2}\in\{\la b \ra, \la a^2ba \ra\}$ then $v_2\in\{av_1, a^2v_1, bav_1, abav_1\}$; $a^2b^k\in \Aut(\G)$ implies that $v_1\sim b^kav_1\iff  a^2b^kv_1\sim v_1$ for $k\in {\{0,1\}}$. Since $aba^2\in G_{v_1}$ we have $a^2bv_1=abav_1$ hence $v_1\sim b^kav_1\iff v_1\sim ab^kav_1$ for $k\in {\{0,1\}}$;
   \item if $G_{v_1}=\la a^2ba\ra,G_{v_2}=\la b\ra, v_2=abv_1$, then $v_1\sim abv_1\iff ba^2v_1\sim v_1\iff ba^2(a^2ba)v_1\sim v_1\iff a^2bv_1\sim v_1\iff bv_1\sim av_1 \iff a^2v_2\sim av_1$,  as $ba^2,a \in \Aut(\G)$,
    \item if $|G_{v_1}|, |G_{v_2}|\in \{4,6\}$, where $\phi(v_1)=v_1,\phi(v_2)\neq v_2$ and $\te{min}\{|G_{v_1}|,|G_{v_2}|\}=4$,
     then $v_2=a^ku$ for some $k\in {\{0,1\}}$ and $u\in V(\G)$ such that $G_u\in\{\la a^2ba\ra,\la ab \ra$\}. Since $v_1\sim u\iff v_1\sim h_1h_2u,\forall h_1\in G_{v_1}, h_2\in G_u$, and $a \in \{h_1h_2\mid h_1\in G_{v_1},h_2\in G_u \}$, we have $v_1\sim u \iff v_1\sim au$,  hence $v_1\sim v_2\iff v_1\sim \phi(v_2)$;
    \item  if $|G_{v_1}|, |G_{v_2}|\in \{4,6\}$, where $\phi(v_1)=a^2v_1,\phi(v_2)=a v_2$ and $\te{min}\{|G_{v_1}|,|G_{v_2}|\}=4$,
     then $v_1\sim v_2\iff v_1\sim h_1h_2v_2,\forall \;h_1\in G_{v_1}, h_2\in G_{v_2}$, and $a^2 \in \{h_1h_2\mid h_1\in G_{v_1},h_2\in G_{v_2} \}$. Hence
     $v_1\sim v_2 \iff  v_1\sim a^2v_2 \iff  a^2v_1\sim av_2$, as $a^2\in \Aut(\G)$;
    \item if
    $\phi(v_1)=a^kv_1,\phi(v_2)=a^kv_2 $ for $k\in{\{1,2\}}$, then $a^k\in \Aut(\G)$ implies that
    $v_1\sim v_2 \iff a^kv_1\sim a^kv_2$,
\end{itemize}
The non-trivial
map
$\phi$ fixes two vertices, $v_1,v_2$, such that
$G_{v_1}=\langle a\rangle$, $ G_{v_2}=\langle a^2b\rangle $, but $G_{v_1}\cap G_{v_2}=\{1\}$; a contradiction.
\end{proof}
\begin{lemma}\label{opqtss}
Let $G=A_4$ and let $\G$ be as in Proposition \ref{arlu} and consider the action of $G$ on $V(\G)$. Then,
there exists no orbit of size 12.
\end{lemma}
\begin{proof}

Suppose that $v\in V(\G)$ and $|\ma v|=12$. If there exists an orbit of size 3, let $w\in V(\G)$ be such that $v\sim aw\iff v\sim a^2w$ and $|\ma w|=3$; otherwise, let $w=v$.
In \cite[Proposition 3.7]{ww}, M. E. Watkins proved that $G$ has no GRR. Arguing in a similar way, we can reach a contradiction.

By possibly replacing $\G$ with its complement, we assume that the
size of the set $\{gv\;{\mathlarger{\mid}}\; v\sim gv,\;g\in G\}$ is bounded by 5. Since $v\sim gv \iff v \sim g^{-1}v$ for $g\in  \Aut(\G)$, $v$ is connected to at most 2 vertices of the set $\{av,abav,abv,a^2bv\}$. 

We consider two cases for the number of neighbours of $v$ in $\{av,abav,abv,a^2bv\}$. First, we let this number be at most equal to 1. Moreover, if $v$ is adjacent to exactly one of $av,abav,abv,a^2bv$, then without loss of generality we assume that this vertex is $av$. 
Let $\phi:V(\G)\rightarrow V(\G)$ be such that 
\[ \phi(u)=\left\{ 
\begin{array}{ll}
      \tau g\tau^{-1} v,\;&\;  \text{  if }u=g\te{v},\;g\in G,\;\te{v}\in\{v,w\},\\
      u, \;&\;\text{  if }G_u=G,
\end{array} 
\right. \]
where $\tau$ is the transition $\tau=(12)\in S_4$, if $v\sim aba^2v \iff v\sim a^2bav$ or, if not, then $\tau$ is the transition $(23)$ or $(13)$, if $v\sim bv \iff v\sim a^2bav$ or $v\sim bv \iff v\sim aba^2v$, respectively.

Let us now consider the case when $v$ is adjacent to two vertices of the set $\{av,abav,abv,a^2bv\}$. We assume that $v$ is adjacent to $av$ and $ abav$; the proof for the other cases is similar. Let us define the map $\psi:V(\G)\rightarrow V(\G)$, where 
\[ 
\psi(u)= \left\{ 
\begin{array}{ll}
      b gb^{-1} \te{v},\;&\; \text{  if }u=g\te{v},\;g\in G,\;\te{v}\in\{v,w\},\\
      u, \;&\;\text{  if }G_u=G.
\end{array} 
\right.\]
 
The maps $\phi$ and $\psi$ are non-trivial  automorphisms of $\G$; the proof is similar to the proof of \cite[Proposition 3.7]{ww}. However, $\phi(v)=\psi(
v)=v$ and $G_v=\{1\}$; a contradiction.
\end{proof}

\begin{proof}[Proof of Proposition \ref{arlu}]
We assumed that $\G$ is a graph on at most 15 vertices having $\Aut(\G)\cong G$.
Moreover, we observed that,
since the action of $G$ on $V(\G)$ is faithful, $b\notin G_z$, for some $z\in V(\G)$. By Lemmas \ref{opqt} and
\ref{opqtss}, there exists no orbit of size 4 or 12, and $|G_z|=6.$
We will show that the map $\chi:V(\G)\rightarrow V(\G)$,  
\[ 
\chi(v)= {\left\{
\begin{array}{ll}
     aba^2v,\;&\;\text{     if }G_v=\la b\ra,\\
     v,\;&\;\text{     if }G_v\neq\la b\ra,
\end{array} 
\right. }
\]
is an automorphism of $\G$. 
Indeed, $v_1\sim v_2\iff \chi(v_1)\sim \chi(v_2)$, for all $v_1,v_2\in V(\G)$, as
\begin{itemize}
\item if $G_{v_1}=G_{v_2}=\la b\ra$ then $v_1\sim v_2\iff aba^2v_1\sim aba^2v_2$, since $ aba^2\in \Aut(\G)$,
    \item if $G_{v_1}\neq\la b\ra, G_{v_2}\neq\la b\ra$, then clearly $v_1\sim v_2\iff \chi(v_1)\sim \chi(v_2)$,
    \item if $G_{v_1}\in {\{\langle b,a^2ba\rangle , G\}}, G_{v_2}=\la b\ra$, then $aba^2\in G_{v_1}$ hence 
    ${v_1\sim v_2\iff v_1\sim aba^2v_2}$,
    \item if $G_{v_1}=\la h\ra$, for some ${h\in \{aba^2, a^2ba\}}$ and $ G_{v_2}=\la b\ra$, then
    $ v_1\sim v_2\iff v_1\sim hv_2 \iff v_1\sim aba^2v_2$, since $h\in G_{v_1}, b\in G_{v_2}$ and     $aba^2\in \{h,hb\}$.
\end{itemize}
The non-trivial
map
$\chi$ fixes two vertices, $v_1,v_2\in \ma z$ such that
$G_{v_1}=\la aba^2\ra,  G_{v_2}=\la a^2ba\ra$, contradicting the property $G_{v_1}\cap G_{v_2}=\{1\}$.
Therefore, $\alpha(G)\geq 16$.

We will show that $\alpha(G)= 16$ by constructing a graph $\G$ on 16 vertices with $\Aut(\G)\cong G$.
\begin{construction}{\label{con100}}
Let $\G$ be a graph with vertex set $V(\G)=\ma 1\cup 
\ma {1'}\cup \ma {1''}$, where
$\ma 1=\{1,2,...,6\},\;\ma {1'}=\{1',2',...,6'\}$ and $\ma {1''}=\{1'',2'',3'',4''\}$.
We define the edge set of $\G$ to be such that,
for $v,w$ $\in \ma 1$, 
\vspace{ -2 mm}\begin{align*}
v\sim w &\iff v\not\equiv w\;\te{(mod }3);
\\v'\sim w' &\iff v\neq w ;
\\
v''\sim w'' &\iff v\neq w 
;
\\v\sim w' &\iff w-v\equiv 2\;\te{(mod }3);
\\
v\sim w''& \iff \\ [v,w]\in\;& \big\{{(1,1),(2,1), (3,1),(1,2),(5,2),(6,2),(2,3),(4,3), (6,3),(3,4),(4,4),(5,4)}\big\};
\\v'\sim w'' & \iff v\nsim w''
.
\end{align*}\vspace{-12 mm}
\end{construction}
Using
the mathematical software GAP \cite{gap} we computed that $\Aut(\G)\cong G$.
\end{proof}

\end{document}